%% file: SIMODS_M163293.tex
\documentclass[final,hidelinks,onefignum,onetabnum]{siamart251216}

\input{ex_shared}

\ifpdf
\hypersetup{
  pdftitle={On the (In)Compatibility between Group Fairness and Individual Fairness},
  pdfauthor={Shizhou Xu and Thomas Strohmer}
}
\fi

\begin{document}

\maketitle

% REQUIRED
\begin{abstract}
We study the compatibility between utility-optimized statistical parity and individual fairness. While individual fairness seeks to treat similar individuals similarly, optimal statistical parity aims to provide similar treatment to individuals who share relative similarity within their respective sensitive groups. Although both fairness perspectives are desirable, they often conflict in practical applications. Our goal in this work is to analyze the existence of this conflict and its potential solution. In particular, we establish sufficient (sharp) conditions for the compatibility between the optimal (post-processing) statistical parity $L^2$ learning and the ($K$-Lipschitz or $(\epsilon,\delta)$) individual fairness requirements. When a conflict arises, we first relax the optimal statistical parity solution to the Pareto frontier, representing the optimal trade-off between $L^2$ error and statistical disparity, and then examine its compatibility with individual fairness. Our analysis identifies regions along the Pareto frontier that satisfy individual fairness constraints. Additionally, we provide individual fairness guarantees for the composition of a trained model with the optimal post-processing step, enabling practitioners to assess the compatibility of post-processed models. This provides practitioners with a principled approach to achieving Pareto-optimal statistical parity while maintaining individual fairness constraints.
\end{abstract}

% REQUIRED

% REQUIRED

\section{Introduction}

Fairness in machine learning has gained increasing attention due to the increasing trend of decision-making and information-sharing based on machine learning or artificial intelligence (AI) assistance in our daily lives. This concern has been further amplified by recent advancements in large language models (LLMs) and generative AI. Ensuring the ethical integrity of such decision-making or information-sharing processes becomes not only crucial for the responsible development of machine learning but also essential for the growth of our society in a healthy form.

There are two fundamental fairness concepts that approach fairness from distinct philosophical perspectives:
\begin{itemize}[leftmargin=20pt,noitemsep,topsep=5pt,parsep=5pt,partopsep=5pt]
\item \textbf{Group fairness} aims to ensure the (conditional) learning outcome to be equal in distribution or statistics across sensitive groups.
\item \textbf{Individual fairness} aims to guarantee that individuals who share similar qualification data will receive similar learning outcomes.
\end{itemize}
Although both fairness concepts are desirable, they can potentially conflict. To see this, consider a learning outcome that does not satisfy the statistical parity (meaning sensitive groups have different learning outcome distributions), it then becomes necessary to move individuals from different sensitive groups in different directions on the learning outcome space, such as $\{0,1\}$ in classification and $\mathbb{R}$ in one-dimensional regression. However, such statistical parity enforcement can often result in a significant violation of individual fairness. Similarly, individual fairness seeks to ensure that individuals with similar qualifications remain close in the learning outcome space. In general, differences in qualification distributions across sensitive groups may propagate through this closeness preservation, potentially leading to disparities in learning outcome distributions and thereby violating group fairness criteria such as statistical parity.

This fundamental conflict raises a key question: \textit{Under what conditions can both group fairness and individual fairness be  satisfied simultaneously?} 

To address this, we adopt statistical parity as our group fairness definition while ensuring utility optimization. Specifically, we define compatibility between statistical parity and individual fairness as the condition where {\em Pareto-optimal statistical parity (defined as Problem \ref{prob:Optimal L2-objective Learning Pareto Frontier}) solutions remain feasible} under additional individual fairness constraints. Conversely, if these constraints exclude Pareto-optimal solutions, we say individual fairness and optimal statistical parity are incompatible. Incorporating utility optimization in our compatibility analysis helps eliminate low-utility predictors, such as constants and white noise, that trivially satisfy both statistical parity and individual fairness but are neither practical nor fair.

To illustrate this inherent conflict, we analyze individual fairness in a noiseless (deterministic) setting for two reasons: (1) The core incompatibility arises in the noiseless setting, and the findings can be extended naturally to noisy settings. (2) Despite the popularity of noise-based approaches to enforcing individual fairness, they effectively blur distinctions in qualification, making predictions less informative and, arguably, unfair. For instance, excessive noise can force statistical parity and individual fairness to coexist by reducing predictions to a random lottery, where qualifications become irrelevant~\cite{fleisher2021s}.

With this formulation, we derive sufficient conditions for the compatibility of individual fairness with Pareto-optimal statistical parity $L^2$-objective learning. In the rest of the work, we define statistical parity as follows:
\begin{defi}[Statistical Parity] \label{d:statistical parity}
Let $\hat{Y}$ be a prediction random variable and $Z$ a sensitive attribute. The tuple $(\hat{Y}, Z)$ satisfies statistical parity if
\[
\hat{Y} \perp Z.
\]
That is, for any measurable sets $A \in \sigma(\hat{Y})$ and $B \in \sigma(Z)$,
\[
\mathbb{P}(\hat{Y} \in A, Z \in B) = \mathbb{P}(\hat{Y} \in A) \mathbb{P}(Z \in B).
\]
\end{defi}
Here, $(\Omega,\mathcal{F},\mathbb{P})$ is a probability space. $\hat{Y}: \Omega \rightarrow \mathcal{Y}$ and $Z: \Omega \rightarrow \mathcal{Z}$ are the random variables (or equivalently measurable functions) that map elements from the underlying probability space $\Omega$ to the state spaces $(\mathcal{Y}, \sigma(\hat{Y}),\mathbb{P} \circ \hat{Y}^{-1})$ and $(\mathcal{Z}, \sigma(Z),\mathbb{P} \circ Z^{-1})$. $\sigma(S)$ denotes the sigma-algebra generated by the random variable $S$ for $S \in \{\hat{Y},Z\}$.

Furthermore, let $\mathcal{P}(\mathcal{S})$ denote the set of probability measures on the spaces $\mathcal{S} \in \{\mathcal{X}, \mathcal{Z}, \mathcal{Y}\}$, $C$ denote a given utility loss function $C: \mathcal{P}(\mathcal{Y}) \times\mathcal{P}(\mathcal{Z}) \to \mathbb{R}_+$, $D$ denote a given statistical parity violation quantification $D: \mathcal{P}(\mathcal{Y}) \times\mathcal{P}(\mathcal{Z}) \to \mathbb{R}_+$, and $\mathcal{F}$ denote a feasible solution set. We define the Pareto-optimal statistical parity by
\begin{defi}[Pareto-optimal statistical parity]
    \begin{equation*}
    \inf_{f \in \mathcal{F}} \{C(\hat{Y},Z): \ \hat{Y} = f(X,Z), \ D(\hat{Y},Z) \leq d \}, \qquad d \in [0,\infty).
\end{equation*}
\end{defi}
We provide formal definitions of $C$, $D$, and $\mathcal{F}$ in Section \ref{S:setting} and Problem \ref{prob:Optimal L2-objective Learning Pareto Frontier}. In particular, we show in subsection \ref{ss:Statistical Parity Enhanced by Utility Optimization} why Pareto-optimal (or utility-optimized) statistical parity is a better group fairness concept than the original statistical parity when taking utility into consideration.

\subsection{Related Work and Our Contributions}

Beginning with the celebrated work ``Fairness through Awareness" \cite{dwork2012fairness}, there is now a sizable body of research studying machine learning solutions for statistical parity. The resulting approaches can be categorized into the following:
\begin{itemize}[leftmargin=20pt,noitemsep,topsep=5pt,parsep=5pt,partopsep=5pt]
\item \textbf{Pre-processing:} The data are deformed before the training step. The goal is to preserve as much information as possible and also keep the deformed data representation independent of the sensitive variable \cite{calmon2017optimized, kamiran2012data, zemel2013learning, xu2022fair}. 
\item \textbf{In-processing:} The fairness definition is quantified and then integrated into the training process by penalizing unfair outcomes \cite{berk2017convex, zafar2017fairness}. It seems to be the natural way in machine learning to add a penalty term related to unfairness quantification. Theoretical guarantee of such machine-learned outcomes in terms of both performance and fairness can be found in works such as \cite{agarwal2019fair}.
\item \textbf{Post-processing:} The definition of fairness is enforced directly on the learning outcome \cite{jiang2020wasserstein, chzhen2020fair, gouic2020projection, silvia2020general, xu2022fair}. This is the most straightforward approach because the methods directly deform the provided learning outcome to satisfy the statistical parity constraint.
\end{itemize}

% relative works to us
In recent years, the connection between probability metric spaces and statistical parity has received significant attention due to the following remarkable results: the optimal fair distribution of supervised learning, such as classification \cite{jiang2020wasserstein, zhao2022inherent, xian2023fair} and regression \cite{chzhen2020fair, gouic2020projection}, can be characterized as the Fr\'echet mean of the learning outcome marginals on the Wasserstein space (or, more generally, metric space of probability measures), which is also known as the Wasserstein barycenter in the optimal transport literature. Thereafter, \cite{xu2022fair} generalizes the post-processing barycenter characterization to all supervised learning models that tend to estimate conditional expectation, including all classification and regression, via a pre-processing (or synthetic data) approach. Furthermore, \cite{xu2022fair} provides a provable Pareto frontier, which extends the barycenter characterization to the provable optimal trade-off between utility loss and any statistical disparity level in both post-processing and pre-processing approaches.

The exploration of compatibility between statistical parity and individual fairness can be traced back to at least \cite{dwork2012fairness}, which highlights the potential conflict between these two fairness concepts and proposes methods to achieve both. Subsequent research has delved into this conflict heuristically \cite{binns2020apparent, chouldechova2018frontiers, corbett2018measure, joseph2016fairness} and experimentally \cite{zhou2022group}. Despite these efforts, the exploration of methods achieving both group and individual fairness remains limited. To our knowledge, only \cite{dwork2012fairness, lohia2019bias, garcia2021maxmin} have endeavored to achieve both group and individual fairness, yet their approaches remain experimental. Our contribution lies in being the first to provide a theoretical analysis of the (in)compatibility and, if compatible, to propose provable methods for achieving the optimal trade-off between statistical parity and individual fairness.

In settings where utility is disregarded, compatibility between statistical parity definitions, such as statistical parity, equalized odds, or accuracy parity, and individual fairness is rarely a concern. For example, individual fairness naturally aligns with accuracy parity when similarly qualified individuals receive similar labels~\cite{zhao2021costs}. Moreover, statistical parity and accuracy parity can both be trivially satisfied by assigning a constant prediction, while equalized odds can be met through random guessing. However, such low-utility yet “definition-compatible” solutions fail to ensure meaningful fairness; much like how excessive noise, which is introduced to enforce individual fairness, can itself lead to substantial unfairness.

In contrast, this work focuses on the (in)compatibility between individual fairness and \emph{utility-optimized} statistical parity, explicitly avoiding trivial compatibility results that lack practical utility. As demonstrated in Section~\ref{ss:Statistical Parity Enhanced by Utility Optimization}, utility-optimized statistical parity represents a stronger group fairness concept than standard statistical parity, as it balances fairness with predictive performance.

In this work, we study the compatibility between the optimal learning outcome for statistical parity and individual fairness in a post-processing setting, where utility is measured by the $L^2$ loss and fairness is defined by statistical parity. Specifically, we analyze the conditions under which the optimal post-processing statistical parity $L^2$ learning (see Problem~\ref{prob:Optimal Fair L2-objective Learning Outcome}) also satisfies K-Lipschitz individual fairness or $(\epsilon,\delta)$ individual fairness. In essence, our goal is to achieve both statistical parity \emph{and} individual fairness while minimizing utility loss.

To end the current section, we summarize the questions that we are targeting in this work and provide the corresponding informal answer.

%See Remark \ref{r:Individual Fairness serves as an Additional Constraint} for why we choose individual fairness as an additional constraint to the the optimal statistical parity $L^2$ learning problem but not the other way around.

\begin{enumerate}[leftmargin=20pt,noitemsep,topsep=5pt,parsep=5pt,partopsep=5pt]
\item \textit{When is the optimal post-processing statistical parity $L^2$ learning compatible with K-Lipschitz individual fairness?}

Theorem \ref{th:post-processing incompatibility lip IF} and Theorem \ref{th:incompatibility lip IF} show that, unless the learning outcome automatically satisfies statistical parity, neither the optimal post-processing statistical parity $L^2$ learning nor the Pareto frontier (any non-trivial optimal trade-off between utility loss and statistical disparity) is compatible with the K-Lipschitz individual fairness definition under mild assumptions. That is, the optimal statistical parity $L^2$ learning and the Pareto frontier has an inherent conflict with the K-Lipschitz individual fairness definition.

\item \textit{When is the optimal post-processing statistical parity $L^2$ learning compatible with $(\epsilon,\delta)$ individual fairness?}

Lemma \ref{l:compatibility optimal fair L2 and IF} shows that, under certain assumptions regarding the relationship between $\epsilon$ and $\delta$, it is possible to achieve both the optimal post-processing statistical parity $L^2$ learning and $(\epsilon,\delta)$ individual fairness.

\item \textit{If the optimal post-processing statistical parity $L^2$ learning and $(\epsilon,\delta)$ individual fairness are not compatible, which portion of the Pareto frontier (or the Pareto optimal solutions) becomes compatible with the $(\epsilon,\delta)$ individual fairness?}

Theorem \ref{th:compatibility Pareto optimal fair L2 and IF} shows that, in the case of incompatibility between optimal statistical parity $L^2$ learning and $(\epsilon,\delta)$ individual fairness, it is guaranteed to have a non-trivial portion of Pareto frontier (between $L^2$ loss and statistical disparity) compatible with $(\epsilon,\delta)$ individual fairness, when $\epsilon < \delta$. That is, $(\epsilon,\delta)$ individual fairness does not affect the Pareto optimality up to a reduction of the statistical disparity level in the learning outcome.

\item \textit{Given that $\hat{Y} = g(X,Z)$ or $g(X)$ for some trained model $g: \mathcal{X} \times \mathcal{Z} \rightarrow \mathcal{Y}$, what guarantees exist regarding the compatibility of the post-processed learning outcome when the trained model is followed by the Pareto-optimal post-processing steps? Furthermore, when applying the Pareto-optimal statistical parity technique as a pre-processing step to generate fair data representation or synthetic data for statistical parity, what compatibility guarantees hold for the pre-processed learning outcome when the trained model is applied after the (Pareto) optimal pre-processing learning steps?}

Theorem \ref{th:composition_epsilon_delta} and Theorem \ref{th:composition_lipschitz} provide compatibility analysis for the learning outcome which composes a trained model with a post-processing step (or composes a pre-processing step with a trained model) under the $(\epsilon,\delta)$ individual fairness and K-Lipschitz individual fairness assumption, respectively. So that researchers and practitioners can obtain compatibility results when composing the proposed post-processing (or pre-processing after modification) with some in-processing methods with individual fairness guarantee.
\end{enumerate}

The remainder of this paper is organized as follows: Section 2 introduces the problem setting and provides essential preliminaries on optimal transport. It also reviews the (Pareto) optimal $L^2$ learning framework that satisfies statistical parity, forming the foundation for our main results. Section 3 examines the compatibility between individual fairness, K-Lipschitz-IF and $(\epsilon,\delta)$-IF, and the optimal statistical parity $L^2$ learning. Section 4 presents a partial remedy: when the optimal $L^2$ fair learning is not fully compatible with a given $(\epsilon,\delta)$-individual fairness requirement, we identify portions of the Pareto frontier that remain compatible. Section 5 establishes compatibility guarantees for both post-processed and pre-processed learning models, considering cases where the Pareto-optimal statistical parity technique is applied either as a post-processing step following a trained model or as a pre-processing step preceding one. Finally, Section 6 presents numerical experiments on real-world datasets to empirically analyze which segments of the Pareto frontier for statistical parity align with specific individual fairness constraints.

\section{Problem Setting: Individual Fairness \& Pareto-optimal $L^2$ Learning}\label{S:setting} In this section, we present the foundational techniques necessary for developing the main (in)compatibility results. Specifically, we begin by introducing and generalizing definitions of individual fairness. We then review the Pareto-optimal statistical parity solutions, which serve as the basis for our fairness analysis. Finally, we justify our problem formulation and formally define the notion of (in)compatibility within this framework.

\subsection{Individual Fairness Definitions} \label{ss:Generalized Individual Fairness Definitions}

Individual fairness requires that \emph{similar individuals receive similar outcomes}. Classical formalizations include $K$-Lipschitz individual fairness (IF)~\cite{dwork2012fairness} and $(\epsilon,\delta)$-IF~\cite{friedler2021possibility}. In our setting, we allow the predictor to depend on the sensitive attribute,
\[
f:\ \mathcal X\times \mathcal Z \to \mathcal Y,
\]
so that fairness corrections may explicitly use $Z$. This choice is motivated by the \emph{Fairness Through Awareness} principle~\cite{dwork2012fairness}, which advocates recognizing sensitive structure in order to correct for it. It is also consistent with the fact that the optimal statistical parity regressor or $L^2$ learning function typically requires $Z$-aware adjustments~\cite{chzhen2020fair,gouic2020projection,xu2022fair}. Related post-processing frameworks may also use group information through conditional probability $\mathbb{P}(Z\mid X)$~\cite{xian2023fair}.

\begin{rema}[Fairness through awareness] \label{r:fairness_through_awareness}
``Fairness through Awareness'' \cite{dwork2012fairness} is mostly referred to as a work on individual fairness. But, at least to our understanding, the fundamental idea behind that work is to leverage the knowledge of sensitive information and unfairness to diminish the unfairness in learning outcomes. To that end, the authors of~\cite{dwork2012fairness} proposed a method applying sensitive information and randomness to achieve both statistical parity and individual fairness, which is also the goal of the present work. In fact, ``Fairness through Awareness'' is one of the main inspirations for our compatibility study.
\end{rema}

Thus, we adopt individual fairness constraints that allow fairness-aware functions to depend on $Z$:

\begin{defi}[Uniform K-Lipschitz-IF] \label{d:Lip IF generalization}
A function $f: \mathcal{X} \times \mathcal{Z} \rightarrow \mathcal{Y}$ satisfies uniform K-Lipschitz Individual Fairness if, for all $x_1, x_2 \in \mathcal{X}$, there exists $K \in \mathbb{R}^{+}$ such that $$ \sup_{z_1,z_2} d_{\mathcal{Y}}(f(x_1,z_1),f(x_2,z_2)) < K d_{\mathcal{X}}(x_1,x_2).$$
\end{defi}

\begin{defi}[Uniform $(\epsilon,\delta)$-IF] \label{d:IF generalization}
A function $f: \mathcal{X} \times \mathcal{Z} \rightarrow \mathcal{Y}$ satisfies uniform $(\epsilon,\delta)$ Individual Fairness if, for all $x_1, x_2 \in \mathcal{X}$, there exists $(\epsilon,\delta) \in (\mathbb{R}^{+})^2$ such that $$d_{\mathcal{X}}(x_1,x_2) < \epsilon \implies  \sup_{z_1,z_2} d_{\mathcal{Y}}(f(x_1,z_1),f(x_2,z_2)) < \delta.$$
\end{defi}
Here, uniform emphasizes that the above Lipschitz conditions apply uniformly across sensitive attributes.

\begin{rema}[Main difference between $(\epsilon,\delta)$-IF and K-Lipschitz-IF]
Notice that uniform K-Lipschitz-IF implies uniform $(\epsilon,K \epsilon)$-IF. Hence, the latter is usually considered to be a relaxed version of the former. The key difference is that the $(\epsilon,\delta)$-IF definition allows different learning outcomes assigned to individuals with the same independent variable value $x$ but different sensitive information $z$. That is, $(\epsilon,\delta)$-IF allows different learning outcome to be assigned to individuals (for example $\{1,2\}$) who share the same qualification ($x_1 = x_2$) but have different sensitive information ($z_1 \neq z_2$): $z_1 \neq z_2 \implies f(x_1,z_1) \neq f(x_2,z_2)$ even if $x_1 = x_2$. This relaxation is compatible with the fundamental idea underlying the optimal statistical parity $L^2$ learning (or more generally fairness through awareness) to let sensitive information dependent functions be applied to the same $x$ to achieve statistical parity. As we will show later, such a relaxation results in different compatibility with the optimal statistical parity $L^2$ learning between the two individual fairness definitions.
\end{rema}

Finally, since we assume that $(\mathcal{Y}, ||\cdot||)$ is a Euclidean space, the distance metric $d_{\mathcal{Y}}$ in our post-processing setting is induced by the Euclidean norm. Consequently, the individual fairness constraints take the following forms:

\begin{itemize}[leftmargin=20pt,noitemsep,topsep=5pt,parsep=0pt,partopsep=0pt]
\item A function $f: \mathcal{Y} \times \mathcal{Z} \rightarrow \mathcal{Y}$ satisfies uniform (post-processing) K-Lipschitz Individual Fairness if, for all $y_1, y_2 \in \mathcal{Y}$, there exists $K \in \mathbb{R}^{+}$ such that  
\[
\sup_{z_1,z_2} ||f(y_1,z_1) - f(y_2,z_2)|| < K ||y_1 - y_2||.
\]
\item A function $f: \mathcal{Y} \times \mathcal{Z} \rightarrow \mathcal{Y}$ satisfies uniform (post-processing) $(\epsilon,\delta)$ Individual Fairness if, for all $y_1, y_2 \in \mathcal{Y}$, there exist $(\epsilon,\delta) \in (\mathbb{R}^{+})^2$ such that  
\[
||y_1 - y_2|| < \epsilon \implies  \sup_{z_1,z_2} ||f(y_1,z_1) - f(y_2,z_2)|| < \delta.
\]
\end{itemize}

These (post-processing) individual fairness definitions are specifically formulated for the post-processing step to better illustrate the inherent conflicts that arise at this stage. Once we analyze the composition of the post-processing step with a trained model or learning outcome, we revert to the individual fairness definitions on $f: \mathcal{X} \times \mathcal{Z} \rightarrow \mathcal{Y}$ with $\mathcal{X} \neq \mathcal{Y}$. To clarify, whenever we focus solely on the post-processing step, we explicitly include ``(post-processing)'' to indicate that fairness constraints and (in)compatibility analysis apply only to this stage.

To formally define the set of functions satisfying $(\epsilon,\delta)$-IF in the post-processing step, we introduce the following admissible function class\footnote{One could also define an admissible set to satisfy the K-Lipschitz-IF constraint. However, as shown later in Theorem \ref{th:incompatibility lip IF}, the optimal statistical parity $L^2$ learning solution (or any non-trivial Pareto optimal solution) is inherently incompatible with the K-Lipschitz-IF constraint both in post-processing and in post-processed results. This intrinsic incompatibility renders further analysis unnecessary, hence we omit the definition of a K-Lipschitz-IF constrained admissible set.}:

\begin{defi}[(Post-processing) $(\epsilon,\delta)$-IF constrained admissible set] \label{d:IF_constrained_admissible_set}
\begin{equation*}
\mathcal{D}_{(\epsilon,\delta)} := \{f \in L^2(\mathcal{Y} \times \mathcal{Z}, \mathcal{Y}) : ||y_1 - y_2|| < \epsilon \implies \sup_{z_1,z_2} ||f(y_1,z_1) - f(y_2,z_2)|| < \delta \}.
\end{equation*}
\end{defi}

For the individual fairness admissible set in the general learning setting, we assume that $\hat{Y} = g(X)$ or $g(X,Z)$ for some measurable function $g: \mathcal{X} \rightarrow \mathcal{Y}$ or $g: \mathcal{X} \times \mathcal{Z} \rightarrow \mathcal{Y}$ representing the trained model. We then define the composition $h := f \circ g$ as
\[
h(X,Z) := f(g(X),Z) \text{ or } f(g(X,Z),Z),
\]
and introduce the corresponding admissible set $\mathcal{D}_{(\epsilon,\delta)}$ as follows:

\begin{defi}[$(\epsilon,\delta)$-IF constrained admissible set] \label{d:IF_constrained_admissible_set}
\[
\mathcal{D}_{(\epsilon,\delta)} := \{f \circ g : d_{\mathcal{X}}(x_1,x_2) < \epsilon \implies \sup_{z_1,z_2} ||f \circ g (x_1,z_1) - f \circ g(x_2,z_2)|| < \delta \}.
\]
\end{defi}
For simplicity, we use the same notation for both definitions, with the specific meaning being clear from the context.

\subsection{(Pareto)-optimal Statistical Parity $L^2$ Learning}

This subsection establishes the theoretical foundation for our compatibility analysis by introducing a quantification of statistical disparity and reviewing the existence and uniqueness of optimal solutions to the optimal post-processing $L^2$ learning for statistical parity and its relaxation to a Pareto frontier: Problem \ref{prob:Optimal Fair L2-objective Learning Outcome} and Problem \ref{prob:Optimal L2-objective Learning Pareto Frontier} below.

\subsubsection{Quantification of Statistical Disparity} \label{ss:Quantification of Statistical Disparity}

Here, we introduce Wasserstein disparity, a statistical disparity quantification that we apply to characterize the Pareto-optimal statistical parity $L^2$ learning. Before defining the statistical disparity quantification or relaxation, we first provide a brief review of Wasserstein space and barycenter and fix the notations used in the present work. Given $\mu, \nu \in \mathcal{P}(\mathbb{R}^d)$ where $\mathcal{P}(\mathbb{R}^d)$ denotes the set of all the probability measures on $\mathbb{R}^d$, $$\mathcal{W}_2(\mu,\nu) := \left(\inf_{ \lambda \in \prod (\mu,\nu) } \Big\{\int_{\mathbb{R}^d \times \mathbb{R}^d} ||x_1 - x_2||^2 d \lambda(x_1,x_2)\Big\}\right)^{\frac{1}{2}}.$$ Here, $\prod (\mu,\nu) := \{\pi \in \mathcal{P}((\mathbb{R}^d)^2): \int_{\mathbb{R}^d} d\pi(\cdot,v) = \mu, \int_{\mathbb{R}^d} d\pi(u,\cdot) = \nu \}$.
$(\mathcal{P}_2(\mathbb{R}^d),\mathcal{W}_2)$ is called the Wasserstein space, where $$\mathcal{P}_2(\mathbb{R}^d):= \Big\{\mu \in \mathcal{P}(\mathbb{R}^d): \int_{\mathbb{R}^d} ||x||^2 d\mu < \infty\Big\}.$$
Given $\{\mu_z\}_{z \in \mathcal{Z}} \subset (\mathcal{P}_2(\mathbb{R}^d),\mathcal{W}_2)$ for some index set $\mathcal{Z}$, their barycenter with weights $\lambda \in \mathcal{P}(\mathcal{Z})$ is 
$$\bar{\mu} := \text{argmin}_{\mu \in \mathcal{P}_2(\mathbb{R}^d)} \Big\{\int_{\mathcal{Z}} \mathcal{W}_2^2(\mu_z,\mu) d\lambda(z)\Big\}.$$
In \cite{agueh2011barycenters, le2017existence}, the authors prove the existence and uniqueness of $\bar{\mu}$ under the assumption that $\lambda(\{\mu_z \in \mathcal{P}_{2,ac}(\mathbb{R}^d)\}) > 0$, where $\mathcal{P}_{2,ac}(\mathbb{R}^d) := \{\mu \in \mathcal{P}_{2}: \mu \ll \mathcal{L} \}$ and $\mathcal{L}$ is the Lebesgue measure on $\mathbb{R}^d$. Therefore, in the rest of the present work, we often assume $\{\mu_z\}_z := \{\mathcal{L}(\hat{Y}_z)\}_z \subset \mathcal{P}_{2,ac}(\mathbb{R}^d)\})$, where $\hat{Y}_z$ denotes the conditional distribution of $\hat{Y}$ on the event $\{Z = z\}$ defined by the disintegration theorem \cite{santambrogio2015optimal}. For an empirical data set $\{(\hat{y}_i,z_i)\}_{i = 1}^{N}$, $\hat{Y}_z$ is the random variable with uniform distribution on $\{\hat{y}_i: z_i = z\}$.

In order to relax the hard statistical parity or independence constraint, we need to introduce a quantification of statistical disparity. To that end, we apply the average pairwise Wasserstein ($\mathcal{W}_2$) distance among (the provided predictions of) sensitive groups to quantify statistical disparity and show its desirable properties.

\begin{defi}[Wasserstein disparity, see also~\cite{xu2022fair}] \label{d:W2 disparity}
$$D(\hat{Y},Z) := \left( \int_{\mathcal{Z}^2} \mathcal{W}^2_2(\mu_{z_1}, \mu_{z_2}) d\lambda^{\otimes 2}((z_1,z_2)) \right)^{\frac{1}{2}},$$
where $\mu_{z_s} := \mathbb{P} \circ \hat{Y}^{-1}_{z_s}$ for $ s \in \{1,2\} $ and $\lambda := \mathbb{P} \circ Z^{-1}$ denote the law or distribution of $\hat{Y}_{z_s}$ and $Z$ respectively.
\end{defi}
(To clarify, we note that $\hat{Y}_{z_s}^{-1}: \sigma(\hat{Y}) \rightarrow \mathcal{F}$ finds the pre-image of an event on the state space $(\mathcal{Y}, \sigma(\hat{Y}))$ in the underlying probability space $(\Omega,\mathcal{F},\mathbb{P})$.) The Wasserstein disparity defined above has the following properties:

\begin{itemize}[leftmargin=20pt,noitemsep,topsep=5pt,parsep=5pt,partopsep=5pt]
\item[1] 
\begin{lem}[Characterization of statistical parity]\label{l:w_disparity_iff_satistical_parity}
    $$D(\hat{Y},Z) = 0 \text{ if and only if } \hat{Y} \perp Z.$$
\end{lem}
See proof in Appendix \ref{a:Section 2 Appendix}.

\item[2] \textbf{(Physics interpretation)} In physics, $D(\hat{Y},Z)$ can be interpreted as the expected minimum amount of work required to remove the distributional discrepancy between two randomly chosen (according to the distribution of $Z$ or $\lambda$) sensitive groups on the provided learning outcome $\hat{Y}$.
\item[3] \textbf{(Characterization of the Pareto frontier)} As shown in \cite{xu2022fair}, by adopting the Wasserstein disparity to relax the hard independence constraint, the optimal trade-off between the $L^2$ loss and statistical disparity (quantified by Wasserstein disparity) is characterized by the geodesic path from the conditional (on the event $\{Z = z\}$) distributions of $\hat{Y}$ to their barycenter on the Wasserstein space. 
\end{itemize}

Due to the listed properties, the average pairwise Wasserstein definition of statistical disparity is a natural quantification when studying the trade-off between statistical disparity and $L^2$ loss.

\subsubsection{Existence and Uniqueness of Optimal Solutions}

Now, we establish the theoretical foundation for our compatibility analysis by summarizing the existence and uniqueness results for the optimal solutions of the Pareto-optimal $L^2$ learning for statistical parity.

Under the $L^2$-objective learning assumption, the goal of post-processing is to minimize the utility sacrifice, which is quantified by the $L^2$ distance between the original outcome $\hat{Y}$ and the post-processed outcome $f(\hat{Y},Z)$. Given a measurable function $g$ such that $\hat{Y} = g(X)$ or $\hat{Y} = g(X,Z)$ with $g: \mathcal{X} \to \mathcal{Y}$ or $g: \mathcal{X} \times \mathcal{Z} \to \mathcal{Y}$, the post-processing function $f: \mathcal{Y} \times \mathcal{Z} \to \mathcal{Y}$ is designed to reduce statistical disparity in $\hat{Y}$. Applying the triangle inequality, we obtain:
\[
    \underbrace{||Y - f(\hat{Y},Z)||_2}_{\text{total loss}} 
    \leq \underbrace{||Y - \hat{Y}||_2}_{\text{training loss}} 
    + \underbrace{||\hat{Y} - f(\hat{Y},Z)||_2}_{\text{post-processing loss}},
\]
where $||Y_1 - Y_2||_2 := \bigl(\int_{\Omega} ||Y_1(\omega) - Y_2(\omega)||^2 \,d\mathbb{P}(\omega)\bigr)^{\frac{1}{2}}$, and $||\cdot||$ denotes the Euclidean norm on $\mathcal{Y}$. Since only $(\hat{Y},Z)$ is available in the post-processing phase, the objective is to minimize the post-processing loss while assuming that the training loss has already been optimized.

\begin{prob}[Optimal Post-processing Statistical Parity $L^2$-Objective Learning]\label{prob:Optimal Fair L2-objective Learning Outcome}
\[
\inf_{f \in L^2(\mathcal{Y} \times \mathcal{Z},\mathcal{Y})} 
\big\{\, ||\hat{Y} - f(\hat{Y},Z)||_2^2 : f(\hat{Y},Z) \perp Z \big\}.
\]
\end{prob}
\noindent
Here, $L^2(\mathcal{Y} \times \mathcal{Z},\mathcal{Y})$ consists of all square-integrable measurable functions from $\mathcal{Y} \times \mathcal{Z}$ to $\mathcal{Y}$, and the constraint $f(\hat{Y},Z)\perp Z$ ensures statistical parity. The function $f(\cdot,z)$ may depend on $Z$, consistent with fairness through awareness \cite{dwork2012fairness}.

If the optimal solution to Problem~\ref{prob:Optimal Fair L2-objective Learning Outcome} is incompatible with individual fairness, we relax the strict statistical parity constraint by allowing a bounded Wasserstein disparity, $D(\hat{Y},Z) \leq d$. This creates a continuum of solutions that characterize the trade-off between utility loss and statistical disparity, forming a Pareto frontier.

\begin{prob}[Post-processing $L^2$-Objective Learning Pareto Frontier]\label{prob:Optimal L2-objective Learning Pareto Frontier}
\[
\inf_{f \in L^2(\mathcal{Y} \times \mathcal{Z},\mathcal{Y})} 
\big\{\, ||\hat{Y} - f(\hat{Y},Z)||_2^2 : D(f(\hat{Y},Z),Z) \leq d \big\}.
\]
\end{prob}
\noindent
Here, $D(\hat{Y},Z)$ represents the Wasserstein disparity (Definition~\ref{d:W2 disparity}), and the parameter $d \in [0,\infty)$ controls the degree of statistical parity relaxation. Notably, when $d=0$, Problem~\ref{prob:Optimal L2-objective Learning Pareto Frontier} reduces to Problem~\ref{prob:Optimal Fair L2-objective Learning Outcome}. As demonstrated in Section~\ref{ss:Quantification of Statistical Disparity}, this formulation captures the full spectrum of trade-offs between $L^2$-norm utility loss and statistical disparity.

We next recall the characterization of the optimal solution to the strict statistical-parity $L^2$ problem. The one-dimensional specialization was characterized in \cite{chzhen2020fair,gouic2020projection}, while the general finite-dimensional Euclidean case was established in \cite{xu2022fair}. Because Problem~\ref{prob:Optimal Fair L2-objective Learning Outcome} allows arbitrary finite-dimensional $\mathcal Y$, we use the corresponding technical lemma from \cite{xu2022fair}. Specifically, let $\mu_z := \mathcal{L}(\hat{Y}_z)$ denote the conditional distribution of $\hat{Y}$ given $Z = z$, and let $T_z: \mathcal{Y} \to \mathcal{Y}$ be the optimal transport map \cite{brenier1991polar} satisfying
\[
(T_z)_{\sharp} \mu_z = \bar{\mu},
\]
where $T_{\sharp}\mu := \mu \circ T^{-1}$ represents the push-forward measure of $\mu$ under the mapping $T$. The following theorem characterizes the unique solution to Problem~\ref{prob:Optimal Fair L2-objective Learning Outcome}:

\begin{lem}[Optimal Fair $L^2$ Learning Characterization \cite{xu2022fair}]\label{l:Optimal Fair $L^2$-Objective Supervised Learning Characterization}
Assume that $\hat{Y}$ has sensitive conditional distributions $\{\mathcal{L}(\hat{Y}_z)\}_{z \in \mathcal{Z}} =: \{\mu_z\}_{z \in \mathcal{Z}} \subset \mathcal{P}_{2,ac}(\mathcal{Y})$. Then, there exists a unique $f^* \in L^2(\mathcal{Y} \times \mathcal{Z},\mathcal{Y})$ defined by
\[
f^*(\cdot,z) := T_z(\cdot), \quad \text{for } \lambda\text{-a.e. } z \in \mathcal{Z},
\]
such that
\[
    ||\hat{Y} - f^*(\hat{Y},Z)||_2^2  
    =  \inf_{f \in L^2} \{||\hat{Y} - f(\hat{Y},Z)||_2^2 : f(\hat{Y},Z) \perp Z\} 
    = \underbrace{\int_{\mathcal{Z}} \mathcal{W}_2^2(\mu_z, \bar{\mu}) d\lambda}_{\text{independence projection loss}}.
\]
\end{lem}

\begin{rema}[Scope of the admissible class and deterministic optimality]
\label{r:deterministic_optimal_solution}
Replacing the admissible set $L^2(\mathcal{Y}\times\mathcal{Z},\mathcal{Y})$ with all measurable functions
$f:\mathcal{Y}\times\mathcal{Z}\to\mathcal{Y}$ does not affect the optimal value, but changes uniqueness to almost sure uniqueness. In practice, neural networks can approximate $\{T_z\}_z$, making square-integrability concerns negligible.

The characterization in Lemma~\ref{l:Optimal Fair $L^2$-Objective Supervised Learning Characterization} is also deterministic. Equivalently, even if one allows randomized post-processing rules represented by Markov kernels, the $L^2$-optimal statistical-parity solution is still induced by deterministic optimal transport maps. Indeed, a randomized rule induces, for each $z$, a coupling between $\mu_z$ and a common output law required by statistical parity. Under the quadratic cost and the absolute-continuity assumptions above, the optimal coupling from $\mu_z$ to the barycenter $\bar\mu$ is uniquely induced by the Brenier map $T_z$ \cite{brenier1991polar}. Thus, randomization does not produce a different $L^2$-optimal solution. Therefore, by the uniqueness of the Brenier map, any nontrivial randomized perturbation of the barycentric map would result in a solution that is strictly sub-optimal with respect to the $L^2$ objective.

Here, we note that the deterministic-optimality conclusion applies to the absolutely continuous population-distribution setting on which this work focuses. For atomic conditional laws, optimal statistical-parity rules need not be deterministic and may be genuinely randomized.
\end{rema}

We define the independence projection loss as the minimum value of Problem~\ref{prob:Optimal Fair L2-objective Learning Outcome}:
\begin{equation}\label{eq:V definition}
    V(\hat{Y},Z) := \min_{f \in L^2(\mathcal{Y} \times \mathcal{Z},\mathcal{Y})} \{||\hat{Y} - f(\hat{Y},Z)||_2^2 :  f(\hat{Y},Z) \perp Z\}.
\end{equation}
This quantity plays a crucial role in characterizing the Pareto frontier, where the geodesic parameter $t := 1 - \frac{d}{\sqrt{2}V}$ governs the trade-off between utility loss and statistical disparity. For a detailed discussion on the geodesic time parameter, we refer interested readers to \cite[Section 4, Theorem 4.1]{xu2022fair}.

If strict statistical parity is not required, the Wasserstein disparity tolerance level $d$ can be nonzero. For $d \in [0,\infty)$, Problem~\ref{prob:Optimal L2-objective Learning Pareto Frontier} characterizes the Pareto frontier between the post-processing $L^2$ utility loss, $||\hat{Y} - f(\hat{Y},Z)||_2^2$, and the residual Wasserstein disparity, $D(f(\hat{Y},Z),Z)$, in the post-processed outcome $f(\hat{Y},Z)$. Let $f^*(\cdot,z)$ denote the solution to Problem~\ref{prob:Optimal Fair L2-objective Learning Outcome} as given in Lemma~\ref{l:Optimal Fair $L^2$-Objective Supervised Learning Characterization}. Define the McCann interpolation~\cite{mccann1997convexity} between the identity map ($\Id$) and the optimal transport map ($f^*(\cdot,z) = T_z$) as $f^*(t)(\cdot,z) := (1-t)\Id + tf^*(\cdot,z)$ for $t \in [0,1]$ and $\lambda$-almost every $z \in \mathcal{Z}$. The following result provides the closed-form, unique solution to Problem~\ref{prob:Optimal L2-objective Learning Pareto Frontier} for any disparity tolerance level $d \in [0,\infty)$.

\begin{lem}[Pareto optimal fair $L^2$-objective learning \cite{xu2022fair}]\label{l:Pareto Optimal Fair L2-objective Learning}
Given $(\hat{Y},Z)$ satisfying $\mu_z \in \mathcal{P}_{2,ac}$, $\lambda$-a.e. and $V$ the independence projection loss defined in \eqref{eq:V definition}, then \begin{equation}
 f_d(\hat{Y},Z) := 
\begin{cases}
    f^*(1- \frac{d}{\sqrt{2}V})(\hat{Y},Z), & \text{if } d \in [0,\sqrt{2}V)\\
    \hat{Y}, & \text{if } d \in [\sqrt{2}V, \infty)
\end{cases}
\end{equation}
are the unique solutions to Problem \ref{prob:Optimal L2-objective Learning Pareto Frontier} for $d \in [0,\infty)$.
\end{lem}

Having established the existence and uniqueness of optimal solutions for Problems~\ref{prob:Optimal Fair L2-objective Learning Outcome} and~\ref{prob:Optimal L2-objective Learning Pareto Frontier}, we now proceed to analyze their compatibility with individual fairness definitions, such as K-Lipschitz-IF and $(\epsilon,\delta)$-IF.

\subsection{Compatibility Definitions \& Problem Setting}

To incorporate individual fairness constraints, we assume an additional admissible set $\mathcal{D}_{(\epsilon,\delta)} \subset L^2(\mathcal{Y}\times\mathcal{Z},\mathcal{Y})$ (see Definition \ref{d:IF_constrained_admissible_set}). Compatibility is then analyzed by comparing each unconstrained problem to its fairness-constrained counterpart:

\begin{prob}[$(\epsilon,\delta)$-IF optimal post-processing $L^2$-learning]\label{prob:IF Constrained Optimal Fair L2-objective Learning}
\[
\inf_{f \in \mathcal{D}_{(\epsilon,\delta)}} 
\big\{\, ||\hat{Y} - f(\hat{Y},Z)||_2^2 : f(\hat{Y},Z)\perp Z \big\}.
\]
\end{prob}

\begin{prob}[$(\epsilon,\delta)$-IF post-processing $L^2$-learning Pareto frontier]\label{prob:IF Constraint Optimal L2-objective Learning Pareto Frontier}
\[
\inf_{f \in \mathcal{D}_{(\epsilon,\delta)}} 
\big\{\, ||\hat{Y} - f(\hat{Y},Z)||_2^2 : D(f(\hat{Y},Z),Z) \leq d \big\}.
\]
\end{prob}

Since Problems~\ref{prob:Optimal Fair L2-objective Learning Outcome} and \ref{prob:Optimal L2-objective Learning Pareto Frontier} relax their respective fairness-constrained versions, their solutions yield lower or equal objective values. We say **compatibility** holds if the optimal solutions coincide, i.e.,
\[
\text{Problem } \ref{prob:Optimal Fair L2-objective Learning Outcome} \;\equiv\; 
\text{Problem } \ref{prob:IF Constrained Optimal Fair L2-objective Learning}, 
\quad\text{or}\quad
\text{Problem } \ref{prob:Optimal L2-objective Learning Pareto Frontier} 
\;\equiv\;
\text{Problem } \ref{prob:IF Constraint Optimal L2-objective Learning Pareto Frontier}.
\]
To complete our compatibility analysis for post-processed trained models, where an arbitrary trained model is followed by the optimal post-processing step for statistical parity, we extend the compatibility definition accordingly. Specifically, we redefine compatibility by replacing $\hat{Y}$ with $g(X)$ and substituting $f \in \mathcal{D}_{(\epsilon,\delta)}$ with $f \circ g \in \mathcal{D}_{(\epsilon,\delta)}$, thereby generalizing our analysis to broader settings.

\begin{rema}[Compatibility analysis beyond post-processing] \label{r:Compatibility analysis for pre-processing}
While our primary focus is on post-processing, the same analytical framework applies to:
\begin{enumerate}[leftmargin=20pt,noitemsep,topsep=5pt,parsep=5pt,partopsep=5pt]
  \item \emph{Unconstrained} statistical parity $L^2$ learning, where no explicit fairness intervention (post-, in-, or pre-processing) is applied, provided that the conditional expectation $\mathbb{E}(Y\mid X,Z)$ is available or can be accurately estimated.
  \item \emph{Pre-processing} techniques that modify $X$ (or $\hat{Y}$) before training to induce statistical parity in the learned representations.
\end{enumerate}
For a comprehensive discussion, we refer to \cite{xu2022fair}. Our primary analysis focuses on post-processing for clarity.
\end{rema}

To conclude this section, we address several fundamental questions underlying our problem setting:
\begin{itemize}[leftmargin=20pt,noitemsep,topsep=5pt,parsep=5pt,partopsep=5pt]
    \item Why do we focus on \emph{(Pareto) optimal} statistical parity $L^2$ learning rather than standard statistical parity?
    \item Why is individual fairness introduced as an \emph{additional} constraint rather than the primary optimization objective?
\end{itemize}
Readers already familiar with utility-optimized statistical parity and individual fairness as a ``defensive'' fairness constraint may proceed directly to Section~3 for the compatibility analysis.

\subsubsection{Statistical Parity Enhanced by Utility Optimization}\label{ss:Statistical Parity Enhanced by Utility Optimization}

In this subsection, we show how utility optimization helps to overcome the major insufficiency of the original statistical parity concept as a group fairness definition. There are three major criticisms of statistical parity as a fairness definition \cite{dwork2012fairness, hardt2016equality}: (1) reduced utility, (2) self-fulfilling prophecy, and (3) subset targeting. But, as shown in \cite{xu2022fair}, the first two of the three insufficiencies can be (partially) fixed by maximizing the utility and adopting the Pareto frontier. 

\begin{itemize}[leftmargin=20pt,noitemsep,topsep=5pt,parsep=5pt,partopsep=5pt]
\item[I] {\bf Reduced utility:} Since statistical parity is an additional constraint on the learning outcome, it necessarily compromises utility. While the utility sacrifice is unavoidable, a Pareto frontier allows practitioners to choose a statistical disparity level that can be tolerated in their own use case and apply the Pareto optimal solution at the chosen tolerance level. In that case, practitioners can trade statistical disparity tolerance for utility and vice versa.
\item[II] {\bf Self-fulfilling prophecy:} The expression `self-fulfilling prophecy' refers to random, careless, or malicious selection in some of the sensitive groups, even in the case statistical parity is satisfied. But utility maximization enforces careful selection in all sensitive groups and therefore prevents random, careless, or malicious selection in any of the sensitive groups from happening, regardless of the size of the sensitive group.
%\item[III] [Subset Targeting] Subset targeting refers to that fact that statistical parity ($\hat{Y} \perp Z$) does not guarantee the subset $\tilde{Y} \subset \hat{Y}$ also satisfies $\tilde{Y} \perp Z$. In \cite{dwork2012fairness}, the author argues that adversaries can target members who shares $\{\hat{Y} = 0\} \cap \{Z = 0\}$ and $\{\hat{Y} = 1\} \cap \{Z = 1\}$ to create subset of $\hat{Y}$ that is strongly correlated to $Z$. But we show that the logic here is a self-loop: The goal of the adversary is to determine $Z$, which is unknown to them. Now, since $\hat{Y} \perp Z$, the adversary cannot extract any information of $Z$ from $\hat{Y}$. Therefore, it is impossible for the adversary to generate $\{\hat{Y} = 0\} \cap \{Z = 0\}$ or $\{\hat{Y} = 1\} \cap \{Z = 1\}$ unless he or she already have the knowledge of $Z$. The best the adversary can do is to perform random sampling to generate a subset $\tilde{Y}$ from $\hat{Y}$. But, as shown in Proposition \ref{p: Sampling Preserves Statistical Parity}, any random sampling actually preserve statistical parity. Therefore, it is impossible for adversaries to sub-targeting provided a prediction already satisfies statistical parity.
\end{itemize}

We refer interested readers to \cite{xu2022fair} for more detailed explanations on how the connection between utility optimization for statistical parity and the optimal multi-marginal matching (equivalently, n-coupling or multidimensional Monge-Kantorovich problem) solves the above two insufficiencies and provides explainability. Also, we refer readers to \cite{gangbo1998optimal} for the technical details of the equivalence between the Wasserstein barycenter problem and the multidimensional Monge-Kantorovich problem.

\begin{rema}[Relative vs absolute similarity] \label{r:Relative vs Absolute Similarity}
Interestingly, by solving the above listed insufficiencies, the utility-maximized statistical parity also provides similar treatment to similar individuals, but the similarity is now defined based on the position in the individual's sensitive group. It is different from the similarity in the individual fairness concept, which is defined based on the individual's position in the whole group. Hence, the compatibility between individual fairness and the (Pareto) optimal (with respective to utility) statistical parity $L^2$-objective learning we study here is, in its nutshell, the compatibility between the following two viewpoints of fairness:
\begin{itemize}[leftmargin=20pt,noitemsep,topsep=5pt,parsep=5pt,partopsep=5pt]
\item \textbf{Relative similarity:} Individuals who share similar positions in his or her own sensitive group should receive similar learning outcomes.
\item \textbf{Absolute similarity:} Individuals sharing similar positions in the entire big group (without considering sensitive groups) should receive similar learning outcomes.
\end{itemize}
\end{rema}

The following example provides an intuition of the relative similarity and how utility optimization solves the two mentioned insufficiencies.

\begin{exam}[College admission]
In a college admission setting, let $X$ represent a student's qualification based solely on their standardized test score at admission, $Y$ be their GPA after four years, $Z$ denote their gender, and $\hat{Y}$ be the predicted GPA. Statistical parity requires the predicted GPA distributions to be identical for male and female students, while individual fairness demands that students with similar test scores receive similar predicted GPAs, irrespective of gender. Assuming an equal ratio of admitted male and female students and that students at the same percentile within each gender receive the same prediction, the optimal post-processing statistical parity $L^2$ learning adjusts predictions to minimize $L^2$ loss while enforcing fairness. Specifically, for each percentile $k$, the predicted GPA for male and female students is first averaged as $\bar{y}(k) := 0.5\hat{y}(m,k) + 0.5\hat{y}(f,k)$, ensuring optimal statistical parity. To introduce a trade-off between statistical parity and utility loss, a Pareto-optimal interpolation is applied, defining the adjusted predictions as $\hat{y}(m,k,t) := (1-t)\hat{y}(m,k) + t\bar{y}(k)$ and $\hat{y}(f,k,t) := (1-t)\hat{y}(f,k) + t\bar{y}(k)$ for $t \in [0,1]$, where $t$ determines the exact position on the Pareto frontier. This method addresses key criticisms of statistical parity: (1) The Pareto frontier provides a principled trade-off, allowing controlled utility sacrifice for fairness, and (2) the structured interpolation ensures that students maintaining the same relative ranking within their gender group receive consistent predictions, preventing arbitrary or unfair adjustments.
\end{exam}

%\begin{open problem}
%The Wasserstein barycenter method and interpretation leverages the fact that $L^2$ spaces enjoys the duality between variance minimization and inner product maximization to ensure relative similarity. What does the relative similarity mean in more general spaces than $L^2$?
%\end{open problem}

\subsubsection{Individual Fairness as an Additional Constraint}\label{ss:justify_if_additional_constraint}

Now, we justify why individual fairness is introduced as an additional constraint in an optimization problem aimed at achieving statistical parity, rather than the other way around.

In our post-processing setting, we assume no access to $X$, $Y$, or the function class $\mathcal{F}$ used during training, where the objective was either $\inf_{f \in \mathcal{F}} {||Y - f(X)||_2^2}$ or $\inf_{f \in \mathcal{F}} {||Y - f(X,Z)||_2^2}$. The only available information is the provided learning outcome $\hat{Y}$ and the sensitive attribute $Z$. Since $\hat{Y}$ serves as the qualification variable for post-processing, modifying it solely for individual fairness lacks justification. Notably, the identity function inherently satisfies both $1$-Lipschitz individual fairness (IF) and $(\epsilon,\epsilon)$-IF, implying that post-processing adjustments are unnecessary if individual fairness were the primary concern.

However, when introducing statistical parity constraints in post-processing, modifications to $\hat{Y}$ may inadvertently violate individual fairness. Consequently, it becomes natural to impose individual fairness as an additional constraint on the optimization problem ensuring statistical parity, rather than using individual fairness as the primary objective. This approach ensures that statistical parity adjustments do not introduce undue harm to individual-level fairness.

For a broader discussion on why individual fairness is best treated as a necessary condition rather than a standalone goal, we refer interested readers to \cite{fleisher2021s}.

\section{Compatibility between the Optimal Statistical Parity $L^2$ Learning and Individual Fairness}

In this section, we study the compatibility between the optimal statistical parity $L^2$-objective learning (the solution to Problem \ref{prob:Optimal Fair L2-objective Learning Outcome}) and the two individual fairness definitions: K-Lipschitz-IF and $(\epsilon,\delta)$-IF.

\subsection{Optimal Statistical Parity $L^2$ Learning and Lipschitz-IF}

To start, we show the inherent incompatibility between the (Pareto) optimal statistical parity $L^2$ learning and the K-Lipschitz-IF definition for any $K > 0$. The incompatibility is due to the strict prevention from applying different functions or maps to the same $y \in \mathcal{Y}$, which unfortunately is necessary to achieve (Pareto) optimality for statistical parity in $L^2$ learning.

\begin{thm}[Incompatibility between optimal statistical parity $L^2$ learning and (post-processing) Lipschitz-IF constraints] \label{th:post-processing incompatibility lip IF}
Under the following two assumptions
\begin{itemize}[leftmargin=20pt,noitemsep,topsep=5pt,parsep=5pt,partopsep=5pt]
\item[1] $\hat{Y} \not\perp Z$ (statistical parity is not automatically satisfied by the prediction $\hat{Y}$),
\item[2] $\dom(f^*(\cdot,z)) = \mathcal{Y},  \lambda-a.e. z \in \mathcal{Z}$ (the optimal fair $L^2$ learning is capable of making predictions for all individuals, even unobserved),
\end{itemize}
neither $f^*$ nor $\{f_d\}_{d \in [0,\sqrt{2}V)}$ is compatible with the K-Lipschitz-IF definition for any $K>0$.
\end{thm}

\begin{proof}
    See Appendix \ref{a:Section 3 Appendix}.
\end{proof}

Remark~\ref{r:deterministic_optimal_solution} shows that allowing randomized post-processing does not create a different optimizer for the original statistical-parity $L^2$ problem. Nevertheless, one may ask whether non-optimal randomized perturbations of the barycentric solution can restore compatibility if pointwise Lipschitz-IF is replaced by expectation-level or distributional variants. The following definition and proposition show that this is not the case for additive independent randomized perturbations.

\begin{defi}[Expectation-level and distributional Lipschitz-IF for randomized post-processing]
\label{d:randomized_lipschitz_IF}
Let $Q:\mathcal Y\times\mathcal Z\to\mathcal P_2(\mathcal Y)$ be a randomized post-processing rule, where $Q_{y,z}$ denotes the output distribution at input $(y,z)$. Define its conditional mean by
\[
m_Q(y,z):=\int_{\mathcal Y} y'\,dQ_{y,z}(y').
\]
We say that $Q$ satisfies expectation-level uniform $K$-Lipschitz-IF if
\[
\sup_{z_1,z_2\in\mathcal Z}
\|m_Q(y_1,z_1)-m_Q(y_2,z_2)\|
\le K\|y_1-y_2\|,
\qquad \forall y_1,y_2\in\mathcal Y.
\]
Given a metric $\mathsf d$ on $\mathcal P_2(\mathcal Y)$, we say that $Q$ satisfies distributional uniform $K$-Lipschitz-IF with respect to $\mathsf d$ if
\[
\sup_{z_1,z_2\in\mathcal Z}
\mathsf d(Q_{y_1,z_1},Q_{y_2,z_2})
\le K\|y_1-y_2\|,
\qquad \forall y_1,y_2\in\mathcal Y.
\]
\end{defi}

\begin{prop}[Randomization around the barycentric solution does not remove the Lipschitz incompatibility]
\label{prop:randomized_barycenter_IF_obstruction}
Assume the assumptions of Theorem~\ref{th:post-processing incompatibility lip IF}, and let
$f^*(y,z)=T_z(y)$ denote the optimal statistical-parity $L^2$ post-processing solution from
Lemma~\ref{l:Optimal Fair $L^2$-Objective Supervised Learning Characterization}. Then, there exist
$z_1,z_2\in\mathcal Z$ and a point $y\in\mathcal Y$ in the common domain of the maps
$T_{z_1}$ and $T_{z_2}$ such that
\[
T_{z_1}(y)\neq T_{z_2}(y).
\]
Consider a randomized perturbation of the barycentric solution defined pointwise by
\[
Q_{y,z}:=\mathcal L\bigl(T_z(y)+\xi_{y,z}\bigr),
\qquad y\in\mathcal Y,\ z\in\mathcal Z,
\]
where $\xi_{y,z}$ is a square-integrable noise variable. It follows that:
\begin{enumerate}[leftmargin=20pt,noitemsep,topsep=5pt]
    \item If $\mathbb E[\xi_{y,z}]=0$ for all $(y,z)$, then $Q$ violates expectation-level uniform $K$-Lipschitz-IF for every finite $K>0$.
    \item If, at the point $y$ above, the added noise has the same distribution across $z_1$ and $z_2$, i.e.,
    \[
    \xi_{y,z_1}\stackrel{d}{=}\xi_{y,z_2},
    \]
    then $Q$ violates distributional uniform $K$-Lipschitz-IF with respect to any metric $\mathsf d$ on $\mathcal P_2(\mathcal Y)$ for every finite $K>0$.
\end{enumerate}
\end{prop}

\begin{proof}
    See Appendix \ref{a:Section 3 Appendix}.
\end{proof}

Consequently, the root of Lipschitz incompatibility is not merely the deterministic nature of the optimal fair $L^2$ solution. Even when we allow randomized solution at the cost of utility, mean-zero perturbations preserve the expectation-level incompatibility, and group-independent perturbations preserve the distributional incompatibility with respect to $\mathcal W_2$.

Now, if we replace $\hat{Y}$ with $g(X)$, then the above result implies the following incompatibility result for the post-processed (for statistical parity) learning outcome.
\begin{thm}[Incompatibility between optimal statistical parity $L^2$ learning and Lipschitz-IF constraints] \label{th:incompatibility lip IF}
Under the following two assumptions
\begin{itemize}[leftmargin=20pt,noitemsep,topsep=5pt,parsep=5pt,partopsep=5pt]
\item[1.] $g(X) \not\perp Z$ (statistical parity is not automatically satisfied by the predictor $g$),
\item[2.] $\dom(f^*(\cdot,z)) = \img(g) \subset \mathcal{Y},  \lambda-a.e. z \in \mathcal{Z}$ (the post-processing step is capable of making predictions for all output of the predictor with all possible sensitive information),
\end{itemize}
neither $h^* := f^* \circ g$ nor $\{h_d\}_{d \in [0,\sqrt{2}V)} := \{f_d \circ g\}_{d \in [0,\sqrt{2}V)}$ is compatible with the K-Lipschitz-IF definition for any $K>0$.
\end{thm}

\begin{proof}
    This follows directly from Theorem \ref{th:post-processing incompatibility lip IF}.
\end{proof}

The above results demonstrate that when the two assumptions in the theorem hold, K-Lipschitz individual fairness (K-Lipschitz-IF) and the optimal statistical parity $L^2$ learning (both the post-processing step itself and the post-processed composition model) are inherently incompatible. To fully establish this conflict, it remains to show that these assumptions are indeed satisfied in most practical and relevant machine learning scenarios.

\begin{rema}\label{r:predictive model assumption practicality}
We discuss the practicality of the two assumptions in Theorem~\ref{th:post-processing incompatibility lip IF} and Theorem~\ref{th:incompatibility lip IF}:

\begin{itemize}[leftmargin=20pt,noitemsep,topsep=5pt,parsep=5pt,partopsep=5pt]
    \item[] \textbf{Assumption 1:} This assumption excludes the trivial case in which statistical parity is inherently satisfied by the provided learning outcome. In such cases, there is neither a need to trade off utility for statistical parity nor a necessity for a post-processing step.
    
    \item[] \textbf{Assumption 2:} Since the primary goal of supervised learning is to make accurate predictions, a well-designed post-processing step (for a fair supervised learning model) or the post-processed model itself should be capable of providing fair predictions for all possible learning outcomes ($\forall y \in \img(g) \text{ or } \mathcal{Y}$) across all sensitive groups ($\forall z \in \mathcal{Z}$).
\end{itemize}
\end{rema}

\subsection{Optimal Fair $L^2$ Learning and $(\epsilon,\delta)$-IF}

Due to the inherent incompatibility result above, we adopt the $(\epsilon,\delta)$-IF definition, which shares the same heuristic concept of individual fairness with the Lipschitz definition but adds more flexibility by allowing the individuals who share the same qualification data to be mapped to different learning outcome due to their different sensitive information.

The following result provides us a straight-forward sufficient condition for the compatibility between the optimal statistical parity $L^2$ learning and the $(\epsilon,\delta)$-IF.

\begin{lem} [Compatibility between optimal statistical parity $L^2$ learning and (post-processing) $(\epsilon,\delta)$-IF] \label{l:compatibility optimal fair L2 and IF}
Assume $\{\mu_z\}_z \subset \mathcal{P}_{2,ac}(\mathcal{Y})$, let $f^*$ be the (unique) solution to Problem \ref{prob:Optimal Fair L2-objective Learning Outcome}, and $L(f^*) := \sup_{(y,z)} ||f(y,z) - y||$, then for all $(\epsilon,\delta) \in (\mathbb{R}^+)^2$ that satisfy $L(f^*) \leq \frac{\delta - \epsilon}{2}$, $f^*$ is the unique solution to
\begin{equation*}
\inf_{f \in \mathcal{D}_{(\epsilon,\delta)}} \{||\hat{Y} - f(\hat{Y},Z)||_2 : f(\hat{Y},Z) \perp Z\}.
\end{equation*}
\end{lem}

The above result shows that, if the solution of Problem \ref{l:Optimal Fair $L^2$-Objective Supervised Learning Characterization} satisfies the assumption $\sup_{(y,z)} ||f^*(y,z) - y|| \leq \frac{\delta - \epsilon}{2}$, then we have
\begin{equation*}
\text{Problem \ref{prob:Optimal Fair L2-objective Learning Outcome}} \equiv \text{Problem \ref{prob:IF Constrained Optimal Fair L2-objective Learning}}
\end{equation*}
which is equivalent to the compatibility between the optimal statistical parity $L^2$ learning and the (post-processing) $(\epsilon,\delta)$-IF requirement.

\begin{proof}
\begin{align*}
||f^*(y_1,z_1) - f^*(y_2,z_2)|| & \leq ||f^*(y_1,z_1) - y_1|| + ||y_1 - y_2|| + ||y_2 - f^*(y_2,z_2)||\\
& \leq 2 (\sup ||f^*(y,z) - y||) + \epsilon \leq 2 \frac{\delta - \epsilon}{2} + \epsilon = \delta
\end{align*}
where the second inequality follows from the assumption.
\end{proof}

\begin{rema}[Sharpness of the assumption upper bound]
The upper bound condition is sharp without any further assumptions on the sensitive marginal distributions of the learning outcome. To construct a counter-example, one can consider a learning outcome with two sensitive Gaussian marginals, $\hat{Y}_i \sim \mathcal{N}(m_i, \sigma)$ for $i \in \{1,2\}$, with the same standard deviation but different means with $||m_1 - m_2|| = \delta$. Now, assume $$\sup_{(y,z)} ||f^*(y,z) - y|| = \frac{\delta - \epsilon}{2} - h,$$ we show that $V$ cannot be achieved with such a restriction for any $h>0$. Indeed, since $\hat{Y}_i's$ are Gaussian with the same standard deviation, their barycenter is also Gaussian with the same standard deviation. Therefore, the optimal transport maps $f^*(y,z = 1) := y + \frac{\mu_2 - \mu_1}{2}$ and $f^*(y,z = 2) := y + \frac{\mu_1 - \mu_2}{2}$ are rigid translations. Hence, $$\{f^*(\cdot,z)\}_z \not\subset \mathcal{D}_{(\epsilon,\delta)}$$ because $f^*$ satisfies $||f^*(y,z) - y|| = \frac{||\mu_1 - \mu_2||}{2} = \frac{\delta}{2}, \forall y \in \mathcal{Y}$ when $\epsilon = 0$. Hence, for any $h > 0$, we have $$V < \inf_{f \in \mathcal{D}_{(\epsilon,\delta)}} \{||\hat{Y} - f(\hat{Y},Z)||_2 : f(\hat{Y},Z) \perp Z\}.$$ That is, if $\sup_{(y,z)} ||f(y,z) - y|| = \frac{\delta - \epsilon}{2} - h$ for any $h > 0$, the optimal solution is not compatible with the $(\epsilon,\delta)$-IF definition. Hence, the bound is sharp.
\end{rema}

\section{Compatibility between Pareto Frontier and $(\epsilon,\delta)$-IF}

In this section, we consider the scenario where the optimal fair $L^2$ learning solution does not satisfy the $(\epsilon,\delta)$-IF requirement. When the $(\epsilon,\delta)$-IF constraint on the admissible maps prevents us from attaining the optimal fair $L^2$ learning solution, a natural remedy is to identify which portion of the Pareto frontier remains compatible with the $(\epsilon,\delta)$-IF requirement. In other words, we investigate under what conditions one must sacrifice Pareto optimality between $L^2$ loss and statistical disparity in order to meet the $(\epsilon,\delta)$-IF constraint.

The following result provides a sufficient condition for determining the segment of the Pareto frontier that is guaranteed to be compatible with the $(\epsilon,\delta)$-IF requirement. Let $f^*$ denote the solution to Problem~\ref{prob:Optimal Fair L2-objective Learning Outcome} as defined in Lemma~\ref{l:Optimal Fair $L^2$-Objective Supervised Learning Characterization}, and define 
\[
L(f^*) := \sup_{(y,z)} ||f^*(y,z) - y||.
\]
Let $V$ be the independence projection loss as defined in equation~\eqref{eq:V definition}, and let $f_d$ denote the Pareto optimal solutions for $d \in [0,\infty)$ as specified in Lemma~\ref{l:Pareto Optimal Fair L2-objective Learning}. Then, we have:

\begin{thm} [Compatible portion of Pareto frontier with $(\epsilon,\delta)$-IF] \label{th:compatibility Pareto optimal fair L2 and IF}
Suppose that $\dom(f^*(\cdot,z)) = \mathcal{Y}$ for $\lambda-a.e.$ $z \in \mathcal{Z},$ then
\begin{equation*}
f_d \text{ is } (\epsilon,\delta)\text{-IF for }
\begin{cases}
    d \in [\sqrt{2}V(1-\frac{\delta - \epsilon}{2L(f^*)}),\infty), & \text{if } \delta - \epsilon \in [0,2L(f^*))\\
    d \in [0,\infty), & \text{if } \delta - \epsilon \in [2L(f^*), \infty)
\end{cases}
\end{equation*}
for all $\epsilon >0$ as $L(f_d) = (1-\frac{d}{\sqrt{2}V})L(f^*) \leq \frac{\delta - \epsilon}{2}$.
\end{thm}

In other words, $\text{Problem \ref{prob:Optimal L2-objective Learning Pareto Frontier}} \equiv \text{Problem \ref{prob:IF Constraint Optimal L2-objective Learning Pareto Frontier}}$ if
\begin{equation*}
\frac{d}{\sqrt{2}V} \in 
\begin{cases}
    [1 - \frac{\delta - \epsilon}{2L(f^*)}, \infty), & \text{if } \delta - \epsilon \in [0,2L(f^*))\\
    [0,\infty), & \text{if } \delta - \epsilon \in [2L(f^*), \infty)
\end{cases}
\end{equation*}
Here, the assumption is identical to Assumption 2 in Theorem~\ref{th:post-processing incompatibility lip IF}, with its practicality discussed in Remark~\ref{r:predictive model assumption practicality}. The result above establishes that if $\delta - \epsilon \in [0,2L(f^*)]$ (respectively, $\delta - \epsilon \in [2L(f^*), \infty)$), then both the Pareto optimal solutions and Pareto optimality, which is ensured by the uniqueness of the solutions, remain compatible with the $(\epsilon,\delta)$-IF constraint when the statistical disparity tolerance level $d$ lies in the interval $[\sqrt{2}V(1 - \frac{\delta - \epsilon}{2L(f^*)}), \infty)$ (respectively, $[0,\infty)$). 

\begin{rema}[Different Cases of $(\epsilon,\delta)$ in Theorem \ref{th:compatibility Pareto optimal fair L2 and IF}]\label{r:implement_different_epsilon_delta_cases}
Here, we discuss different cases of $(\epsilon,\delta)$ in the above result:
\begin{itemize}[leftmargin=20pt,noitemsep,topsep=5pt,parsep=5pt,partopsep=5pt]
\item[I] [$\delta - \epsilon < 0$] The above result provides sufficient conditions for compatibility without any further assumptions on the distribution of $\hat{Y}_z$. When  $\delta - \epsilon < 0$, let $\epsilon = 0$, $(\epsilon,\delta)$-IF constraint requires $\sup_{z_1,z_2} ||f(y,z_1) - f(y,z_2)|| \leq \delta < 0,$ which leads to a contradiction. Hence, there is no sufficient condition to guarantee any portion of the Pareto frontier to be compatible with $(\epsilon,\delta)$-IF constraint in this case.
\item[II] [$\delta - \epsilon = 0$] In the case of $\delta = \epsilon$, the $(\epsilon,\delta)$-IF constraint does not allow any different maps to the same $y \in \mathcal{Y}$ when set $\epsilon = 0$. Therefore, the inherent incompatibility of K-Lipschitz-IF also applies to $(\epsilon,\delta)$-IF in this case. One has to set the $d$ bigger than or equal to $\sqrt{2}V$ to tolerate the original prediction $\hat{Y}$. Any non-trivial trade-off is prohibited.
\item[III] [$\delta - \epsilon \in (0,2K)$] When $\delta - \epsilon \in (0,2K)$, the $(\epsilon,\delta)$-IF constraint now allows sensitive information dependent maps, but the flexibility is not large enough to tolerate the case where one wants to trade a large amount of utility for statistical parity. Therefore, only a limited portion of the Pareto frontier starting from the original learning outcome is compatible with the $(\epsilon,\delta)$-IF constraint in this case.
\item[IV] [$\delta - \epsilon \in [2K, \infty)$] When $\delta - \epsilon \in [2K, \infty)$, the optimal transport maps that result in the optimal statistical parity $L^2$ learning are compatible with the $(\epsilon,\delta)$-IF constraint. Therefore, all the tolerance levels are compatible in this case.
\end{itemize}
\end{rema}

\section{Composition Results} \label{s:composition_results}

Since a post-processed model is the composition of a trained model and a post-processing step, it is crucial in practice to analyze its compatibility with individual fairness constraints. Therefore, in this section, we establish a compatibility guarantee for the composed model, or equivalently the post-processed model, leveraging the analytical results developed for the post-processing step in the previous sections.

In particular, we assume the trained model, $g: \mathcal{X} \times \mathcal{Z} \rightarrow \mathcal{Y}$ (or $g: \mathcal{X} \rightarrow \mathcal{Y}$ when considering $g(x,z)$ remains constant when $z$ varies), is provided. Here, $(\mathcal{X},d_{\mathcal{X}})$ is the qualification or dependent variable (metric) space and $(\mathcal{Y},||\cdot||)$ is the independent variable (Euclidean) space. That is, the provided prediction $\hat{Y}$ has the form $\hat{Y} = g(X,Z)$, where $X$ is the qualification or dependent variable and $Z$ is the sensitive variable. Now, by assuming $g$ satisfies uniform $(\epsilon,\delta)$-IF: $$d_{\mathcal{X}}(x_1,x_2) < \epsilon \implies \sup_{(z_1,z_2) \in \mathcal{Z}^2}||g(x_1,z_1) - g(x_2,z_2)|| < \delta,$$ or uniform K-Lipschitz-IF: $$\sup_{(z_1,z_2) \in \mathcal{Z}^2}||g(x_1,z_1) - g(x_2,z_2)|| \leq K d_{\mathcal{X}}(x_1,x_2),$$ we provide compatibility guarantee to the post-processed model $f^* \circ g$ or $f_d \circ g$. Here, $f^*$ and $f_d$ are defined as in Lemma \ref{l:Optimal Fair $L^2$-Objective Supervised Learning Characterization} and Lemma \ref{l:Pareto Optimal Fair L2-objective Learning}. The post-processing composition $f \circ g: \mathcal{X} \times \mathcal{Z} \rightarrow \mathcal{Y}$ is defined by $(x,z) \rightarrow f(g(x,z),z)$ for $f \in \{f^*,f_d\}$, where $g$ is the trained model that use $(x,z)$ or $x$ for prediction and $f$ is the post-processing step that applied $z$-dependent maps $f(\cdot,z)$ to deform $g(x,z)$ to satisfy statistical parity.

To prove the main compatibility result, we need the following lemma, which shows the compatibility guarantee for the composition ($f \circ g (x,z) := f(g(x,z),z)$) of an arbitrary measurable function $g: \mathcal{X} \times \mathcal{Z} \rightarrow \mathcal{Y}$ and an arbitrary measurable function $f: \mathcal{Y} \times \mathcal{Z} \rightarrow \mathcal{Y}$.

\begin{lem}[Composition guarantee for $(\epsilon,\delta)$-IF functions] \label{l:composition_1}
Assume $g: \mathcal{X} \times \mathcal{Z} \rightarrow \mathcal{Y}$ satisfies uniform $(\epsilon,\delta)$-IF and define $L(f) := sup_{(y,z)} ||f(y,z) - y||$, define the post-processing composition $f \circ g$ by $(x,z) \rightarrow f \circ g (x,z) := f(g(x,z),z)$, then
\begin{equation*}
    f \circ g \text{ satisfies uniform } (\epsilon,\delta + 2L(f))\text{-IF}, \text{ for all }\epsilon > 0.
\end{equation*}
Furthermore, assume $g: \mathcal{X} \rightarrow \mathcal{Y}$ satisfies $(\epsilon,\delta)$-IF and define the pre-processing composition $g \circ f$ by $(x,z) \rightarrow g \circ f (x,z) := g(f(x,z))$, then
\begin{equation*}
    g \circ f \text{ satisfies uniform } (\epsilon - 2L(f), \delta)\text{-IF}, \text{ for all }\epsilon > 2L(f)
\end{equation*}
\end{lem}

\begin{rema}[Composition analysis for pre-processing]
    As mentioned in Remark \ref{r:Compatibility analysis for pre-processing}, one can apply $f^*$ or $f_d$ as pre-processing steps which now are functions from $\mathcal{X} \times \mathcal{Z}$ to $\mathcal{X}$. See detailed explanation in \cite{xu2022fair}. Therefore, despite our focus on the post-processing case in the present work, Lemma \ref{l:composition_1}, \ref{l:composition_2} and Theorem \ref{th:composition_epsilon_delta}, \ref{th:composition_lipschitz} all include results for both post-processing ($f \circ g$) and pre-processing ($g \circ f$). Therefore, the composition results also include the pre-processing cases so that practitioners or researchers who are interested in pre-processing or synthetic data approaches to both statistical parity and individual fairness can modify the compatibility analysis and apply the respective composition result.
\end{rema}

The above result shows that: (1) For post-processing, if the trained model $g$ satisfies $(\epsilon,\delta)$-IF, then it is guaranteed that the post-processed learning outcome $f \circ g$ satisfies $(\epsilon,\delta + 2L(f))$-IF. (2) For pre-processing, if the trained model $g$ satisfies $(\epsilon,\delta)$-IF for some $\epsilon > 2L(f)$, then it is guaranteed that the pre-processed learning outcome $f \circ g$ satisfies $(\epsilon - 2L(f),\delta)$-IF.

In practical applications, we typically first define the $(\epsilon,\delta)$-IF requirement and then seek Pareto-optimal solutions that achieve both individual fairness and reduced statistical disparity with minimal utility loss. The following result identifies which segments of the Pareto frontier, equivalently which Pareto-optimal solutions, when composed with a trained model $g$, ensure that the post-processed model remains compatible with the given $(\epsilon,\delta)$-IF constraint.

\begin{thm}[Composition guarantee for $(\epsilon,\delta)$-IF trained model]\label{th:composition_epsilon_delta}
    For any $(\epsilon,\delta) \in (\mathbb{R^+})^2$, if $g: \mathcal{X} \times \mathcal{Z} \rightarrow \mathcal{Y}$ satisfies uniform $(\epsilon,\delta_g)$-IF for some $\delta_g < \delta$, then $$\{f_d \circ g\}_{d \in [\sqrt{2}V[1-(\frac{\delta - \delta_g}{2L(f)} \wedge 1)],\infty)} \text{ satisfy uniform } (\epsilon,\delta)\text{-IF}.$$
    Furthermore, if $g: \mathcal{X} \rightarrow \mathcal{Y}$ satisfies $(\epsilon_g,\delta)$-IF for some $\epsilon_g > \epsilon$, then $$\{g \circ f_d\}_{d \in [\sqrt{2}V[1-(\frac{\epsilon_g - \epsilon}{2L(f)} \wedge 1)],\infty)}  \text{ satisfy uniform } (\epsilon,\delta)\text{-IF}.$$
\end{thm}

That is, if we require the post-processed learning outcome to satisfy uniform $(\epsilon,\delta)$-IF, the above result together with Theorem \ref{th:composition_epsilon_delta} shows that one sufficient approach is to first require the trained model $g$ to satisfy $(\epsilon,\delta_g)$-IF for some $\delta_g < \delta$ and then pick $$d \in [\sqrt{2}V[1-(\frac{\delta - \delta_g}{2L(f)} \wedge 1)],\infty)$$ so that the composed Pareto optimal solution $f_d \circ g$ satisfies is guaranteed to satisfy $(\epsilon,\delta)$-IF. On the other hand, if we require the pre-processed learning outcome to satisfy $(\epsilon,\delta)$-IF, the above result together with Theorem shows that one sufficient approach is to first require the trained model $g$ to satisfy $(\epsilon_g,\delta)$-IF for some $\epsilon_g > \epsilon$ and then pick $$d \in [\sqrt{2}V[1-(\frac{\epsilon_g - \epsilon}{2L(f)} \wedge 1)],\infty)$$ so that the composed Pareto optimal solution $g \circ f_d$ satisfies is guaranteed to satisfy $(\epsilon,\delta)$-IF.

Similarly, we have compatibility when assuming $g$ satisfies K-Lipschitz-IF:

\begin{lem}[Composition guarantee for $K-Lipschitz$-IF functions] \label{l:composition_2}
Assume $g: \mathcal{X} \times \mathcal{Z} \rightarrow \mathcal{Y}$ satisfies uniform K-Lipschitz-IF and define $L(f) := sup_{(y,z)} ||f(y,z) - y||$, then we have
\begin{equation*}
    f \circ g \text{ satisfies uniform } (\epsilon,K\epsilon + 2L(f))\text{-IF}, \forall \epsilon > 0.
\end{equation*}
Furthermore, assume $g: \mathcal{X} \rightarrow \mathcal{Y}$ satisfies K-Lipschitz-IF, then
\begin{equation*}
    g \circ f \text{ satisfies uniform } (\epsilon - 2L(f), K\epsilon)\text{-IF}, \forall \epsilon > 2L(f)
\end{equation*}
\end{lem}

Therefore, we are able to derive the following theorem to provide compatibility guarantee for the post-processed model as in the $(\epsilon,\delta)$-IF assumption case above.

\begin{thm}[Composition guarantee for K-Lipschitz-IF trained model]\label{th:composition_lipschitz}
    For any $(\epsilon,\delta) \in (\mathbb{R^+})^2$, if if $g: \mathcal{X} \times \mathcal{Z} \rightarrow \mathcal{Y}$ satisfies uniform K-Lipschitz-IF for some $K \in (0,\frac{\delta}{\epsilon}]$, then $$\{f_d \circ g\}_{d \in [\sqrt{2}V[1-(\frac{\delta - K \epsilon}{2L(f)} \wedge 1)],\infty)}  \text{ satisfy uniform } (\epsilon,\delta)\text{-IF}.$$
    Furthermore, if $g: \mathcal{X} \rightarrow \mathcal{Y}$ satisfies K-Lipschitz-IF for some $K \in \mathbb{R}^+$, then $$\{g \circ f_d\}_{d \in [\sqrt{2}V[1-(\frac{\frac{\delta}{K} - \epsilon}{2L(f)} \wedge 1)],\infty)} \text{ satisfy uniform } (\epsilon,\delta)\text{-IF}.$$
\end{thm}

\section{Empirical Study: Fair Supervised Learning}\label{s:Numerics}

In this section, we present an experimental study on the compatibility analysis developed in the previous sections, using various benchmark data sets.\footnote{The code and results for our experiments are available at: \url{https://github.com/xushizhou/compatibility_group_individual_fairness}.}

\subsection{Benchmark Data and Comparison Methods}

Table~\ref{tab:datasets} provides a summary of the data sets, tested methods, and input/output dimensions.

\begin{table}[htbp]
\centering
\renewcommand{\arraystretch}{1.2}
\begin{tabular}{|l|l|c|c|c|}
\hline
\textbf{Data Set} & \textbf{Methods} & \textbf{Data Size} & \textbf{dim($\mathcal{X}$)} & \textbf{dim($\mathcal{Y}$)} \\
\hline
LSAC  & Linear Regression, ANN & 20,454 & 9  & 1  \\
\hline
CRIME & Linear Regression, ANN & 1,994  & 97 & 1  \\
\hline
CRIME & Linear Regression, ANN & 1,994  & 87 & 11 \\
\hline
\end{tabular}
\caption{Summary of benchmark data sets and test configurations.}
\label{tab:datasets}
\end{table}

For the \textbf{univariate supervised learning} test, we implement the compatibility result established in Theorem~\ref{th:compatibility Pareto optimal fair L2 and IF} using the optimal affine estimation of the Wasserstein barycenter proposed in \cite{xu2022fair}, where the provided prediction $\hat{Y}$ is generated using both linear regression and artificial neural networks (ANN). A key advantage of employing affine estimation for the optimal transport maps is the ability to accurately derive $L(f_{\text{affine}})$, leveraging the fact that affine transport maps decompose into a rigid translation and a linear transformation:
\begin{align*}
L(f_{\text{affine}}) & := \sup_{y,z} ||f_{\text{affine}}(y,z) - y|| \leq \sup_z ||m_z - \bar{m}|| + \sup_{(y,z)} ||f_{\text{linear}}(y,z) - y||\\
& \leq \sup_z ||m_z - \bar{m}|| + \big(\sup_z ||f_{\text{linear}}(\cdot,z) - \Id||_{op}\big)\big(\sup_{y \in \mathcal{Y}} ||y||\big).
\end{align*}
Here, $m_z := \int_{\mathcal{Y}} y d\mu_z$, $\bar{\mu} := \int_{\mathcal{Y}} y d\bar{\mu}$, $||\cdot||_{op}$ denotes the operator norm. $\sup_z ||m_z - \bar{m}||$ is the worst-case rigid translation, and $(\sup_z ||f_{\text{linear}}(\cdot,z) - \Id||_{op})(\sup_{y} ||y||)$ is the worst-case linear deformation. In the experiments, we use the empirical mean to estimate $\sup_z ||m_z - \bar{m}||$, the largest absolute value of eigenvalue for operator norm, and the largest $||y||$ in the sample to estimate $\sup_{y \in \mathcal{Y}}||y||$. In practice, one should derive $\sup_{y \in \mathcal{Y}}||y||$ based on the specific application context.

In order to show the estimation accuracy of the optimal affine maps (``supervised learning name + post-proc. Pareto frontier Est. or Pseudo-barycenter"), we also include the exact cumulative distribution function matching method in \cite{chzhen2020fair} (``supervised learning name + chzhen").

For the \textbf{multi-variate supervised learning} test we also implement the optimal affine method to estimate the post-processing Wasserstein barycenter, estimate the $K$ via the above upper bound, and finally provide the compatibility portion of the Pareto frontier under different $(\epsilon,\delta)$-IF requirements. Here, the provided prediction $\hat{Y}$ is also generated by linear regression and ANN.

For the benchmark data sets, we use CRIME and LSAC. The CRIME dataset contains social, economic, law enforcement, and judicial data for U.S.\ communities in 1994 (1994 examples) \cite{redmond2002data}. In univariate $L^2$ learning, the goal is to predict the number of crimes per $10^5$ population using the remaining dataset information. The sensitive variable is race, specifically whether the percentage of African American population exceeds 30.

In multivariate supervised learning on the CRIME dataset, we retain the same sensitive variable (race). However, the learning task is to predict a vector representing local housing and rental market data, including low-quartile occupied home value, median home value, high-quartile home value, low-quartile rent, median rent, high-quartile rent, median gross rent, number of immigrants, median number of bedrooms, number of vacant households, and the number of crimes.

The LSAC dataset comprises social, economic, and personal data from 20,454 law school students \cite{wightman1998lsac}. In univariate modeling, the objective is to predict students' GPA using the remaining dataset variables. In this context, race serves as the sensitive variable describing whether the student is white/non-white.

\subsection{Numerical Results}

The univariate supervised learning test results for the LSAC and CRIME datasets are presented in the first and second row of Figure~\ref{f:comparison}, respectively. In these plots, the vertical axis represents the $L^2$ loss, while the horizontal axis represents the Wasserstein ($\mathcal{W}_2$) disparity. A result positioned closer to the lower-left corner indicates better performance. In each row, the portion of the Pareto frontier that is guaranteed to be compatible with the displayed $(\epsilon,\delta)$-IF constraint is highlighted in a different color.

We adopt two supervised learning methods: linear regression and artificial neural networks (ANN) with four stacked layers. Each of the first three layers contains of 32 units with Rectified Linear Unit (ReLu) activation, and the final layer consists of one unit with linear activation.

In Figure~\ref{f:comparison}, all three rows illustrate the compatible portion of the Pareto frontier corresponding to the $(\epsilon, \epsilon + 0.3 \cdot 2L(f^*))$, $(\epsilon, \epsilon + 0.6 \cdot 2L(f^*))$, and $(\epsilon, \epsilon + 0.9 \cdot 2L(f^*))$ individual fairness constraints for the LSAC, CRIME, and multivariate CRIME experiments, respectively. Here, the values of $L(f^*)$ are $0.959$, $1.045$, and $3.396$ for linear regression predictions, and $1.265$, $1.385$, and $5.765$ for ANN predictions in the LSAC, CRIME, and multivariate CRIME experiments, respectively. The resulting compatible portion corresponds to the first $30\%$, $60\%$, and $90\%$ of the Pareto frontier. In general, each percentage increase in $\frac{\delta - \epsilon}{2L(f^*)}$ results in one percentage increase in the compatible portion of the Pareto frontier until reaching its endpoint, which represents the estimated optimal post-processing statistical parity $L^2$ learning solution. Additionally, we use the notation $(\epsilon, \epsilon + (\delta - \epsilon))$ in the plots because the compatibility guarantee holds for all $\epsilon \in [0, \infty)$.

\begin{figure}[H]
\centering
\includegraphics[width=\textwidth]{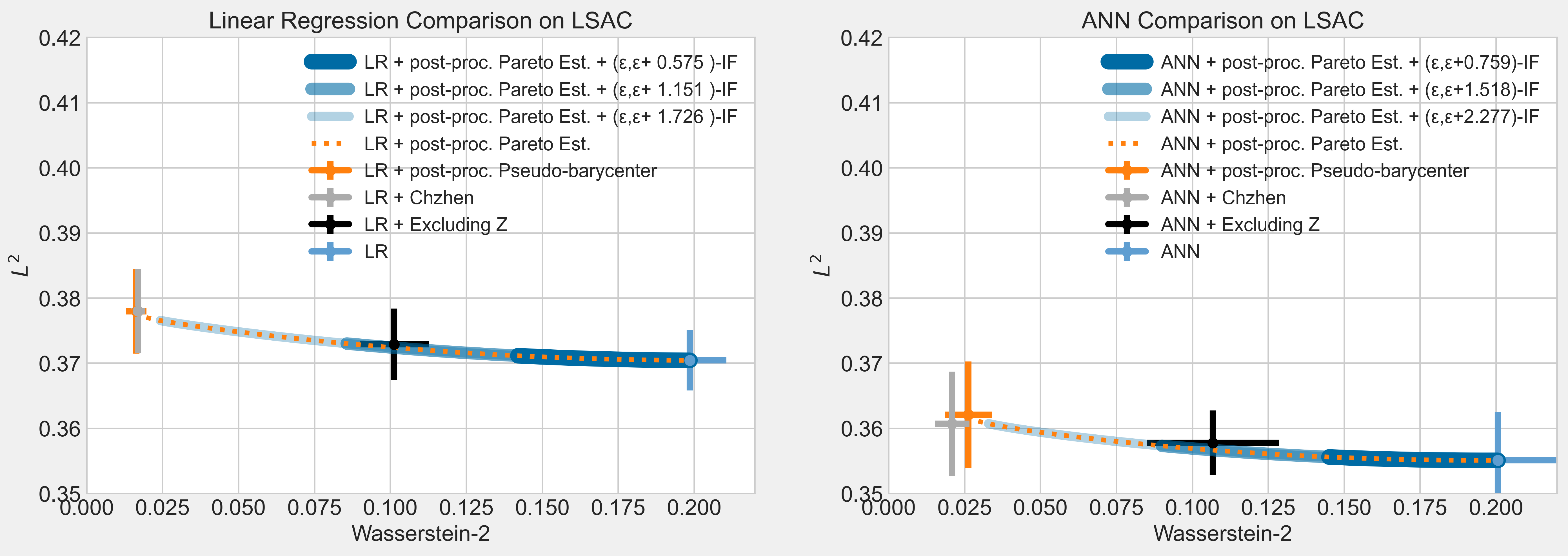}\hfill
\includegraphics[width=\textwidth]{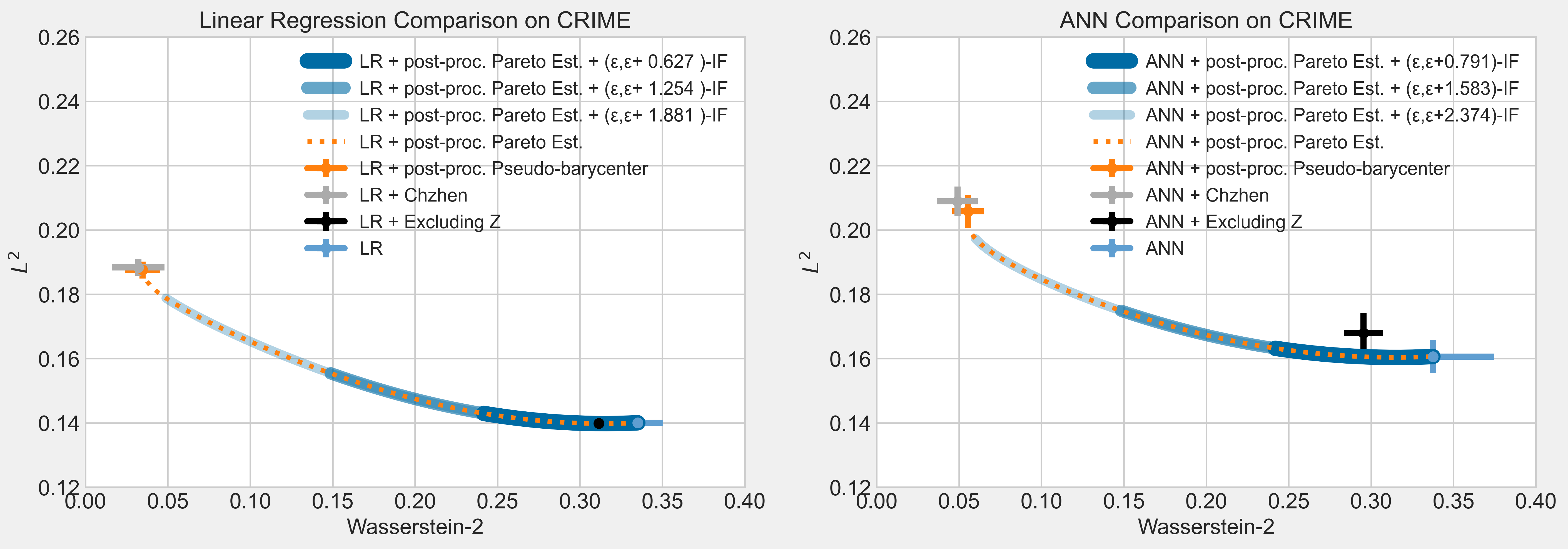}\hfill
\includegraphics[width=\textwidth]{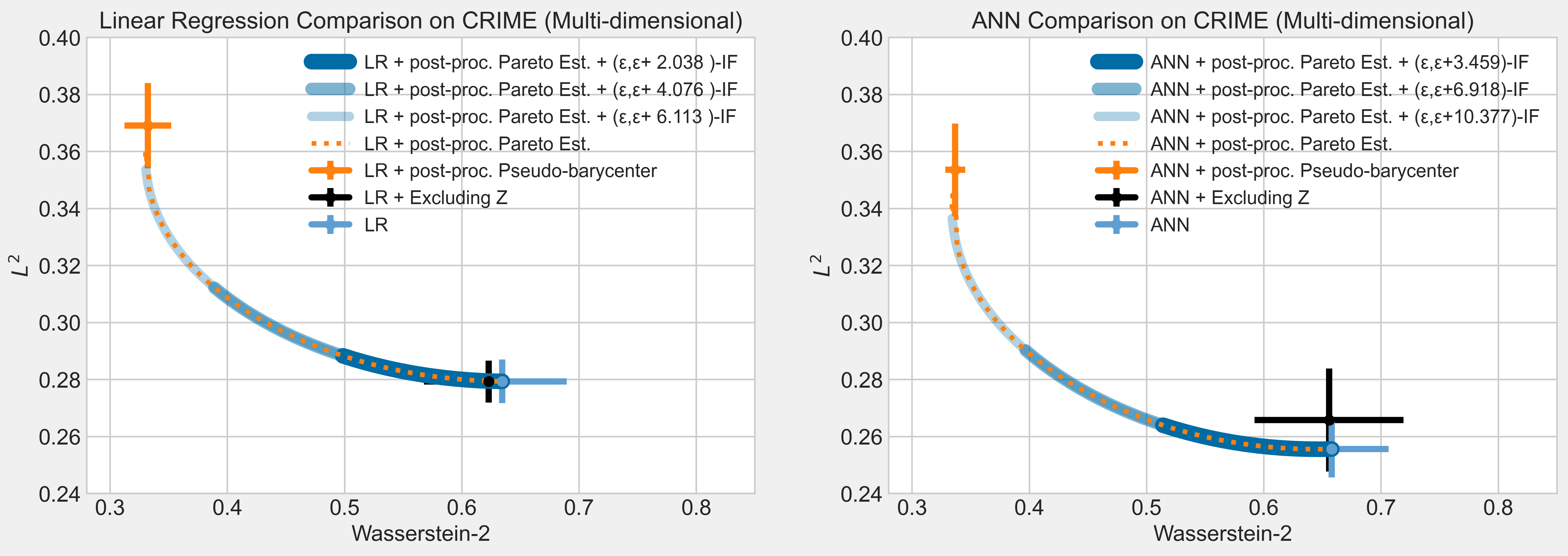}\hfill
\caption{All the three rows above consist of the $L^2$ loss and Wasserstein disparity of the original prediction (LR or ANN), the prediction using data excluding $Z$ (LR or ANN + Excluding Z), the exact post-processing $\mathcal{W}_2$ barycenter via cumulative distribution functions (cdfs) matching approach (LR or ANN + chzhen2020fair), the optimal affine estimation of the post-processing $\mathcal{W}_2$ barycenter (LR or ANN + post-proc. Pseudo-barycenter), the Pareto frontier estimated by the optimal affine maps (LR or ANN + post-proc. Pareto Est.), and finally the portion of the estimated Pareto frontier that is compatible with the corresponding $(\epsilon,\delta)$-IF constraints. Here, $L(f^*) = 0.959, 1.045, 3.396$ for linear regression prediction and $L(f^*) = 1.250, 1.385, 5.765$ for ANN prediction on LSAC, CRIME, and multivariate CRIME respectively. For each $(\epsilon,\delta)$-IF constraint, the compatible portion is the first $(\delta-\epsilon)/(2L(f^*))$ part of the Pareto frontier. More generally, each percentage increase in $(\delta-\epsilon)/(2L(f^*))$ results in one percentage larger portion of the Pareto frontier to be compatible. Also, the portion is guaranteed to satisfy$(\epsilon, \epsilon + (\delta - \epsilon))$-IF for all $\epsilon \in [0,\infty)$}
\label{f:comparison}
\end{figure}

Furthermore, since linear regression satisfies Lipschitz-IF constraint, we can apply the composition result (Theorem \ref{th:composition_lipschitz}) to provide an individual fairness guarantee for the composed learning outcome. In our experiment on LSAC, it follows from the K-Lipschitz condition for $K= 0.1265$ of the linear regression model and the composition theorem, we are able to conclude that the post-processed linear regression model corresponding to the top, middle, and bottom post-processing step satisfies $(\epsilon,0.1265\epsilon+0.3*2L(f^*)), (\epsilon,0.1265\epsilon+0.6*2L(f^*)), (\epsilon,0.1265\epsilon+0.9*2L(f^*))$ individual fairness constraint respectively.

Moreover, if the goal is to make the post-processed linear regression model satisfy a fixed $(\epsilon,\delta)$-IF, then it follows from Theorem \ref{th:composition_lipschitz} that the Pareto optimal solutions $f_d$ are guaranteed to satisfy the required $(\epsilon,\delta)$-IF constraint for all $d \in [\sqrt{2}V(1-(\frac{\frac{\delta}{0.1265} - \epsilon}{2L(f^*)} \wedge 1)),\infty)$ where $V = 0.0739$ and $0.0842$ for linear regression and ANN respectively.

The compatibility results on CRIME can be concluded analogously, except now $L(f^*) = 1.045$ and $V = 0.146$ for linear regression, $L(f^*) = 1.385$ and $V = 0.150$ for ANN prediction, and the linear model satisfies $K=1.0297$-Lipschitz-IF.

The results of the multivariate supervised learning test are shown in the third row of Figure~\ref{f:comparison}. The compatibility result can be concluded to be analogous to the one for the LSAC test. Except now $L(f^*) = 3.396$ and $V = 0.278$ for linear regression, which satisfies $K=1.4566$-Lipschitz condition, and $L(f^*) = 5.765$ and $V = 0.281$ for ANN.

\newpage

\appendix

\section{Appendix: Proofs of the Results in Section 2} \label{a:Section 2 Appendix}

\subsection{Proof of Lemma \ref{l:w_disparity_iff_satistical_parity}}

\begin{proof}
We first show that (1) $\mathcal{W}_2^2(\mathcal{L}(\hat{Y}_z),\bar{\mu}) = 0 \text{, } \lambda\text{-a.e.}\iff \hat{Y} \perp Z,$ then complete the proof by showing (2) $D(\hat{Y},Z) = 0 \iff \mathcal{W}_2^2(\mathcal{L}(\hat{Y}_z),\bar{\mu}) = 0 \text{, } \lambda\text{-a.e.}.$ Here, $\bar{\mu}$ denotes the Wasserstein barycenter of the marginal distributions $\{\mathcal{L}(\hat{Y}_z)\}_z$. 

(1) To start, assume $\hat{Y} \perp Z$, we have $\mathcal{L}(\hat{Y}_z) = \mathcal{L}(\hat{Y})$ $\lambda$-a.e.. That implies $\mathcal{W}_2^2(\mathcal{L}(\hat{Y}_z), \mathcal{L}(\hat{Y})) = 0$, $\lambda$-a.e.. Then, it follows from the uniqueness of $\bar{\mu}$ that $\bar{\mu} = \mathcal{L}(\hat{Y})$ and $\mathcal{W}_2^2(\mathcal{L}(\hat{Y}_z),\bar{\mu}) = 0 \text{, } \lambda\text{-a.e.}.$ For the other direction, assume that $\mathcal{W}_2^2(\mathcal{L}(\hat{Y}_z),\bar{\mu}) = 0 \text{ } \lambda\text{-a.e.}$, then for all $A \in \sigma(\hat{Y})$ and $B \in \sigma(Z)$, we have
\begin{align*}
\mathbb{P}(\{\hat{Y} \in A\} \cap \{Z \in B\}) &= \int_{B} \mathbb{P}(\{\hat{Y}_z \in A\}) d\mathbb{P} \circ Z^{-1}\\
&= \int_{B} \bar{\mu}(A) d\mathbb{P} \circ Z^{-1}\\
&= \bar{\mu}(A) \int_{B} d\mathbb{P} \circ Z^{-1}\\
&= \mathbb{P}(\{\hat{Y} \in A\}) \mathbb{P}(\{Z \in B\}).
\end{align*}
This proves $\hat{Y} \perp Z$ and completes the proof for the first claim.

(2) Assume that $\hat{Y}$ is not independent of $Z$, then $\mathcal{W}_2^2(\mathcal{L}(\hat{Y}_z),\bar{\mu}) \neq 0 \text{, } \lambda\text{-a.e.}$. This implies
\begin{align*}
& \lambda(\{ \mathcal{W}_2^2(\mathcal{L}(\hat{Y}_z),\bar{\mu}) > 0 \}) > 0\\
\iff & \lambda(\bigcup_{k = 1}^{\infty} \{ \mathcal{W}_2^2(\mathcal{L}(\hat{Y}_z),\bar{\mu}) \geq \frac{1}{k} \}) > 0\\
\implies & \exists K < \infty \text{ such that } \lambda(\{ \mathcal{W}_2^2(\mathcal{L}(\hat{Y}_z),\bar{\mu}) \geq \frac{1}{K} \}) > 0\\
\implies & \int_{\mathcal{Z}} \mathcal{W}_2^2(\mathcal{L}(\hat{Y}_z),\bar{\mu}) d\lambda \geq \frac{1}{K} \lambda(\{ \mathcal{W}_2^2(\mathcal{L}(\hat{Y}_z),\bar{\mu}) \geq \frac{1}{K} \}) > 0.
\end{align*}
For the other direction, assume that $\hat{Y} \perp Z$, then it follows from the first claim that $\mathcal{W}_2^2(\mathcal{L}(\hat{Y}_z),\bar{\mu}) = 0 \text{, } \lambda\text{-a.e.}$ which further implies $\int_{\mathcal{Z}} \mathcal{W}_2^2(\mathcal{L}(\hat{Y}_z),\bar{\mu}) d\lambda = 0$. This proves the second claim. \end{proof}

\section{Appendix: Proofs of the Results in Section 3} \label{a:Section 3 Appendix}

\subsection{Proof of Theorem \ref{th:post-processing incompatibility lip IF}}

\begin{proof}
(Optimal $L^2$ learning) It follows from the assumption $\hat{Y} \not\perp Z$ that there exists an event $A \in \mathcal{B}_{\mathcal{Z}}$ such that $\lambda(A) > 0$ and $\mathcal{W}_2^2(Y_{z_1},Y_{z_2}) \neq 0$ for some $z_1, z_2 \in A$.  By Lemma \ref{l:Optimal Fair $L^2$-Objective Supervised Learning Characterization}, the optimal fair $L^2$ learning requires the set of optimal transport maps $\{f(\cdot, z)\}_z$ to map each sensitive marginal learning outcome to their Wasserstein barycenter:
\begin{equation}
f(Y_z,z) = \bar{Y}, \text{ for } \lambda-a.e.  z \in \mathcal{Z}.
\end{equation}
Since $\mathcal{W}_2^2(Y_{z_1},Y_{z_2}) \neq 0$, we have $$\max{ \{\mathcal{W}_2^2(Y_{z_1},\bar{Y}), \mathcal{W}_2^2(Y_{z_2},\bar{Y} \}} > 0,$$ which further implies $f(\cdot,z_1) \neq f(\cdot,z_2)$. Now, it follows from the second assumption that there exits $y \in \mathcal{Y} =  \dom(f(\cdot,z_1)) \bigcap \dom(f(\cdot,z_2))$ such that $$||f(y,z_1) - f(y,z_2)||_2 > 0 = K||y - y||_2,$$ for any $K > 0$. That contradicts the Lipschitz individual fairness definition and hence completes the proof for the incompatibility between the optimal $L^2$ learning and the K-Lipschitz-IF.

(Pareto optimal solutions) By Lemma \ref{l:Pareto Optimal Fair L2-objective Learning}, the Pareto optimal solutions are characterized by the McCann interpolations: $f(1 - \frac{d}{\sqrt{2}V})(\cdot,z)$. But it follows from the assumptions that there exists $y \in \mathcal{Y} =  \dom(f(\cdot,z_1)) \bigcap \dom(f(\cdot,z_2))$ such that $||f(y,z_1) - f(y,z_2)||_2 > 0$, which further implies that $||f(1 - \frac{d}{\sqrt{2}V})(y,z_1) - f(1 - \frac{d}{\sqrt{2}V})(y,z_2)||_2 > 0$ for any $d \in [0, \sqrt{2}V)$. That is, $f(1 - \frac{d}{\sqrt{2}V})$ satisfies the K-Lipschitz-IF constraint only if $d \in [\sqrt{2}V, \infty)$. But the Pareto optimal solution set becomes $\{\hat{Y}\}$ for $d \in [\sqrt{2}V, \infty)$. That completes the proof.
\end{proof}

\subsection{Proof of Proposition~\ref{prop:randomized_barycenter_IF_obstruction}}

\begin{proof}
By the proof of Theorem~\ref{th:post-processing incompatibility lip IF}, there exist
$z_1,z_2\in\mathcal Z$ and a point $y\in\mathcal Y$ in the common domain of the maps
$T_{z_1}$ and $T_{z_2}$ such that
\[
T_{z_1}(y)\neq T_{z_2}(y).
\]
We first prove the expectation-level claim. Assume that $\mathbb E[\xi_{y,z}]=0$ for all $(y,z)$. Then the conditional mean of the randomized output is
\[
m_Q(y,z)
=
\int_{\mathcal Y} y'\,dQ_{y,z}(y')
=
\mathbb E[T_z(y)+\xi_{y,z}]
=
T_z(y).
\]
At the same domain point $y$ and sensitive values $z_1,z_2$, we therefore have
\[
\|m_Q(y,z_1)-m_Q(y,z_2)\|
=
\|T_{z_1}(y)-T_{z_2}(y)\|
>0.
\]
However, the right-hand side of the expectation-level uniform Lipschitz-IF inequality at $y_1=y_2=y$ is
\[
K\|y-y\|=0.
\]
Hence,
\[
\sup_{z,z'\in\mathcal Z}
\|m_Q(y,z)-m_Q(y,z')\|
\le K\|y-y\|
\]
fails for every finite $K>0$.

Next, we prove the distributional claim. Assume that $\xi_{y,z_1}$ and $\xi_{y,z_2}$ have the same distribution. Let $\rho_y$ denote this common noise law. Then
\[
Q_{y,z_1}
=
(\tau_{T_{z_1}(y)})_\sharp \rho_y,
\qquad
Q_{y,z_2}
=
(\tau_{T_{z_2}(y)})_\sharp \rho_y,
\]
where $\tau_a(u):=a+u$ is translation by $a$. Since $\rho_y\in\mathcal P_2(\mathcal Y)$, both translated laws have finite first moments. Moreover,
\[
\int x\,dQ_{y,z_1}(x)
=
T_{z_1}(y)+\int u\,d\rho_y(u),
\qquad
\int x\,dQ_{y,z_2}(x)
=
T_{z_2}(y)+\int u\,d\rho_y(u).
\]
Since $T_{z_1}(y)\neq T_{z_2}(y)$, these two means are different. Hence,
\[
Q_{y,z_1}\neq Q_{y,z_2}.
\]
Therefore, for any metric $\mathsf d$ on $\mathcal P_2(\mathcal Y)$,
\[
\mathsf d(Q_{y,z_1},Q_{y,z_2})>0
\]
by the positivity property of metrics. On the other hand, the right-hand side of the distributional uniform Lipschitz-IF inequality at $y_1=y_2=y$ is
\[
K\|y-y\|=0.
\]
Thus,
\[
\sup_{z,z'\in\mathcal Z}
\mathsf d(Q_{y,z},Q_{y,z'})
\le K\|y-y\|
\]
fails for every finite $K>0$. That completes the proof.
\end{proof}

\section{Appendix: Proofs of the Results in Section 4} \label{a:Section 4 Appendix}

\subsection{Proof of Theorem \ref{th:compatibility Pareto optimal fair L2 and IF}}

\begin{proof}
[$\delta - \epsilon \in [0,2K)$] It suffices to show that the optimal transport map $\{f(t)(\cdot,z)\}_z := \{(1-t)\Id + t f(\cdot,z)\}_z$ is a subset of $\mathcal{D}_{(\epsilon,\delta)}$ given $t = 1 - \frac{d}{\sqrt{2}V}$. Indeed, when $\delta - \epsilon \in [0,2K)$, we have
\begin{align*}
\frac{d}{\sqrt{2}V} \in [1 - \frac{\delta - \epsilon}{2K}, \infty) & \iff 1 - \frac{d}{\sqrt{2}V} \leq \frac{\delta - \epsilon}{2K}\\
& \implies (1 - \frac{d}{\sqrt{2}V}) \sup_{y,z}||f(y,z) - y|| \leq \frac{\delta- \epsilon}{2}\\
& \implies \sup_{y,z} ||(1 - \frac{d}{\sqrt{2}V})f(y,z) + \frac{d}{\sqrt{2}V} y - y|| \leq \frac{\delta- \epsilon}{2}\\
& \iff \sup_{y,z} ||f(1 - \frac{d}{\sqrt{2}V})(y,z) - y|| \leq \frac{\delta- \epsilon}{2}
\end{align*}
Here, the last line is due to McCann interpolation: $(1 - \frac{d}{\sqrt{2}V})f(y,z) + \frac{d}{\sqrt{2}V} y = f(1 - \frac{d}{\sqrt{2}V})(y,z)$. Now, by Lemma \ref{l:Pareto Optimal Fair L2-objective Learning}, $\{f(1 - \frac{d}{\sqrt{2}V})(\cdot,z)\}_z$ are the Pareto optimal solutions at the tolerance level $d$, it follows from the proof of Lemma \ref{l:compatibility optimal fair L2 and IF} that the corresponding Optimal solutions satisfy the $(\epsilon,\delta)$-IF constraint for all $d$ satisfies $\frac{d}{\sqrt{2}V} \in [1 - \frac{\delta - \epsilon}{2K}, \infty)$. That completes the proof for the case of $\delta - \epsilon \in [0,2K)$.

[$\delta - \epsilon \in [2K, \infty)$] When $\delta - \epsilon \in [2K, \infty)$, we have $$K \leq \frac{\delta - \epsilon}{2}.$$ It again follows from the proof of Lemma \ref{l:compatibility optimal fair L2 and IF} that $\{f(1 - \frac{d}{\sqrt{2}V})(\cdot,z)\}_z$ satisfies the $(\epsilon,\delta)$-IF constraint for all $d \in [0,\infty)$. We are done.
\end{proof}

\section{Appendix: Proofs of the Results in Section 5} \label{a:Section 5 Appendix}

\subsection{Proof of Lemma \ref{l:composition_1}}

\begin{proof}
    ($f \circ g$) For the post-processing case, let $(\mathcal{X}, d_{\mathcal{X}}$ be a metric space, $(\mathcal{Y}, ||\cdot||)$ be a Euclidean space. Now, choose $x_1,x_2 \in \mathcal{X}$ to satisfy $d_{\mathcal{X}}(x_1,x_2) < \epsilon$, it follows from the assumption of $g$ that $||g(x_1,z_1) - g(x_2,z_2)|| < \delta$. But we also have
    \begin{align*}
        & ||f(g(x_1,z_1),z_1) - f(g(x_2,z_2),z_2)||\\
        \leq & ||f \circ g (x_1,z_1) - g(x_2,z_2)|| + ||g(x_1,z_1) - g(x_2,z_2)|| + ||g(x_2,z_2) - f \circ g (x_2,z_2)||\\
        < & L + \delta + L = \delta + 2L(f)
    \end{align*}
    where $f \circ g (x,z) := f(g(x,z),z)$. Since our choice of $x_1,x_2 \in \mathcal{X}$ satisfying $d_{\mathcal{X}}(x_1,x_2) < \epsilon$ is arbitrary, we have $f \circ g$ is $(\epsilon,\delta + 2L(f))$-IF by definition. Finally, since our choice of $\epsilon > 0$ is arbitrary, it follows that $f \circ g \text{ is } (\epsilon,\delta + 2L(f))\text{-IF}, \forall \epsilon > 0$.

    ($g \circ f$) For the pre-processing case, let $(\mathcal{X}, ||\cdot||)$ be a Euclidean space, and let $(\mathcal{Y}, d_{\mathcal{Y}}$ be a metric space. Also, define $g \circ f$ by $(x,z) \rightarrow g(f(x,z))$. Now, choose $x_1,x_2 \in \mathcal{X}$ to satisfy $||x_1 - x_2|| < \epsilon - 2L(f)$ for some $\epsilon > 2L(f)$, then we have
    \begin{align*}
        & ||f(x_1,z_1) - f(x_2,z_2)||\\
        \leq & ||f(x_1,z_1) - x_1|| + ||x_1 - x_2|| + ||x_2 - f(x_2,z_2)||\\
        < & L(f) + \epsilon - 2L(f) + L(f) = \epsilon
    \end{align*}
    It follows from the assumption of $g$ that $$d_{\mathcal{Y}}(g(f(x_1,z_1)) - g(f(x_2,z_2))) < \delta.$$ Since our choice of $x_1,x_2 \in \mathcal{X}$ satisfying $||x_1 - x_2|| < \epsilon - 2L(f)$ is arbitrary, that proves $f \circ g$ is $(\epsilon - 2L(f),\delta)$-IF by definition. Finally, since our choice of $\epsilon > 2L(f)$ is arbitrary, it follows that $f \circ g \text{ is } (\epsilon - 2L(f),\delta)\text{-IF}, \forall \epsilon > 2L(f)$.
\end{proof}

\subsection{Proof of Lemma \ref{l:composition_2}}

\begin{proof}
The proof is similar to the proof of Lemma \ref{l:composition_1} above and left to the reader.
\end{proof}

\subsection{Proof of Theorem \ref{th:composition_epsilon_delta}}

\begin{proof}
    [Post-processing $f_d \circ g$] Let $(x_1,x_2) \in \mathcal{X} \times \mathcal{X}$ satisfy $d_{\mathcal{X}}(x_1,x_2) < \epsilon$. It follows from the assumption of $(\epsilon,\delta_g)$-IF of $g$ and the triangle inequality that
    \begin{align*}
        & ||f_d \circ g(x_1,z_1) - f_d \circ g(x_2,z_2)||\\
        \leq & ||f_d \circ g(x_1,z_1) - g(x_1,z_1)|| + ||g(x_1,z_1) - g(x_2,z_2)||  + ||g(x_2,z_2) - f_d \circ g(x_2,z_2)||\\
        < & 2L(f_d) + \delta_g
    \end{align*}
    Now, if $\frac{\delta - \delta_g}{2L(f^*)} \geq 1$, then $2L(f_d) + \delta_g \leq 2L(f^*) + \delta_g \leq \delta, \forall d \in [0,\infty)$. This implies that $||f_d \circ g(x_1,z_1) - f_d \circ g(x_2,z_2) ||< \delta$ for all $d \in [0,\infty)$. On the other hand, if $\frac{\delta - \delta_g}{2L(f^*)} < 1$, then for all $d \in [\sqrt{2}V(1 - \frac{\delta - \delta_g}{2L(f^*)}),\infty)$ we have
    \begin{align*}
        \frac{d}{\sqrt{2V}} \geq 1 - \frac{\delta - \delta_g}{2L(f^*)}
        \implies & 1 - \frac{d}{\sqrt{2V}} \leq \frac{\delta - \delta_g}{2L(f^*)}\\
        \implies & 2L(f^*)(1 - \frac{d}{\sqrt{2V}}) \leq \delta - \delta_g\\
        \implies & 2L(f_d) + \delta_g \leq \delta\\
        \implies & ||f_d \circ g(x_1,z_1) - f_d \circ g(x_2,z_2)|| < \delta.
    \end{align*}
    Here, the second last line follows from $(1 - \frac{d}{\sqrt{2}V})L(f^*) = L(f_d).$

    [Pre-processing $g \circ f_d$] Let $(x_1,x_2) \in \mathcal{X} \times \mathcal{X}$ satisfy $d_{\mathcal{X}}(x_1,x_2) < \epsilon$. Now, if $\frac{\epsilon_g - \epsilon}{2L(f^*)} \geq 1$, then $$2L(f_d) + \epsilon \leq 2L(f^*) + \epsilon \leq \epsilon_g, \forall d \in [0,\infty).$$ Hence, $||f_d(x_1) - f_d(x_2)|| < \epsilon + 2L(f_d) \leq \epsilon_g$ and it follows from the assumption of $(\epsilon_g,\delta)$-IF of $g$ that $d(g \circ f_d(x_1),g \circ f_d(x_2)) < \delta$. On the other hand, if $\frac{\epsilon_g - \epsilon}{2L(f^*)} < 1$, then for all $d \in [\sqrt{2}V(1 - \frac{\epsilon_g - \epsilon}{2L(f^*)}),\infty)$, we have
    \begin{align*}
        \frac{d}{\sqrt{2V}} \geq 1 - \frac{\epsilon_g - \epsilon}{2L(f^*)}
        \implies & 1 - \frac{d}{\sqrt{2V}} \leq \frac{\epsilon_g - \epsilon}{2L(f^*)}\\
        \implies & 2L(f^*)(1 - \frac{d}{\sqrt{2V}}) \leq \epsilon_g - \epsilon\\
        \implies & 2L(f_d) + \epsilon \leq \epsilon_g\\
    \end{align*}
    Hence, $||f_d(x_1) - f_d(x_2)|| < \epsilon + 2L(f_d) \leq \epsilon_g$ and it follows from the assumption of $(\epsilon_g,\delta)$-IF of $g$ that $d(g \circ f_d(x_1),g \circ f_d(x_2)) < \delta$. This completes the proof.

\end{proof}

\subsection{Proof of Theorem \ref{th:composition_lipschitz}}

\begin{proof}
The proof is similar to the proof of Theorem \ref{th:composition_epsilon_delta} above and left to the reader.
\end{proof}

\newpage

\bibliographystyle{siamplain}
\bibliography{references}
\end{document}

%% file: ex_shared.tex
\usepackage{amsmath, amssymb}
\usepackage{hyperref}
\hypersetup{hidelinks}
\usepackage{tabularx}
\usepackage{algorithmic}
\usepackage{mathrsfs}
\usepackage{pgfplots}
\usepackage[skip=2pt,font=scriptsize]{caption}
\usepackage{float}
\usepackage{bbm}
\usepackage{amsfonts}
\usepackage{graphicx}
\usepackage{epstopdf}
\usepackage{lipsum}
\usepackage{enumitem}
\setlist{nosep}
\usepackage [english]{babel}

\usepackage{bm}

\newcolumntype{b}{X}
\newcolumntype{s}{>{\hsize=.5\hsize}X}
\newcolumntype{t}{>{\hsize=.3\hsize}X}

\ifpdf
  \DeclareGraphicsExtensions{.eps,.pdf,.png,.jpg}
\else
  \DeclareGraphicsExtensions{.eps}
\fi

\newcommand{\edit}[1]{\textcolor{blue}{#1}}

\newsiamremark{remark}{Remark}
\newsiamremark{hypothesis}{Hypothesis}
\crefname{hypothesis}{Hypothesis}{Hypotheses}
\newsiamthm{claim}{Claim}
\newsiamthm{prob}{Problem}
\newsiamthm{thm}{Theorem}
\newsiamthm{lem}{Lemma}
\newsiamthm{prop}{Proposition}
\newsiamthm{rema}{Remark}
\newsiamthm{coro}{Corollary}
\newsiamthm{defi}{Definition}
\newsiamthm{exam}{Example}

\headers{Group Fairness and Individual Fairness}{Shizhou Xu and Thomas Strohmer}

\title{On the (In)Compatibility between Group Fairness and Individual Fairness\thanks{This work was funded by NSF DMS-2027248, NSF DMS-2208356, NSF CCF-1934568, and NIH R01HL16351.}}

\author{Shizhou Xu\thanks{Department of Mathematics, University of California Davis 
  (\email{shzxu@ucdavis.edu}).}
\and Thomas Strohmer\thanks{Department of Mathematics, University of California Davis
  (\email{strohmer@math.ucdavis.edu}).}
  }

\usepackage{amsopn}

\DeclareMathOperator{\dom}{Domain}
\DeclareMathOperator{\img}{Image}
\DeclareMathOperator{\Id}{Id}

%% file: references.bib
@article{agueh2011barycenters,
  title={Barycenters in the {W}asserstein space},
  author={Agueh, Martial and Carlier, Guillaume},
  journal={SIAM Journal on Mathematical Analysis},
  volume={43},
  number={2},
  pages={904--924},
  year={2011},
  publisher={SIAM}
}

@article{berk2017convex,
  title={A convex framework for fair regression},
  author={Berk, Richard and Heidari, Hoda and Jabbari, Shahin and Joseph, Matthew and Kearns, Michael and Morgenstern, Jamie and Neel, Seth and Roth, Aaron},
  journal={arXiv preprint arXiv:1706.02409},
  year={2017}
}

@article{brenier1991polar,
  title={Polar factorization and monotone rearrangement of vector-valued functions},
  author={Brenier, Yann},
  journal={Communications on pure and applied mathematics},
  volume={44},
  number={4},
  pages={375--417},
  year={1991},
  publisher={Wiley Online Library}
}

@article{calmon2017optimized,
  title={Optimized pre-processing for discrimination prevention},
  author={Calmon, Flavio and Wei, Dennis and Vinzamuri, Bhanukiran and Natesan Ramamurthy, Karthikeyan and Varshney, Kush R},
  journal={Advances in neural information processing systems},
  volume={30},
  year={2017}
}

@article{chzhen2020fair,
  title={Fair regression with {W}asserstein barycenters},
  author={Chzhen, Evgenii and Denis, Christophe and Hebiri, Mohamed and Oneto, Luca and Pontil, Massimiliano},
  journal={Advances in Neural Information Processing Systems},
  volume={33},
  pages={7321--7331},
  year={2020}
}

@inproceedings{dwork2012fairness,
  title={Fairness through awareness},
  author={Dwork, Cynthia and Hardt, Moritz and Pitassi, Toniann and Reingold, Omer and Zemel, Richard},
  booktitle={Proceedings of the 3rd innovations in theoretical computer science conference},
  pages={214--226},
  year={2012}
}

@inproceedings{fleisher2021s,
  title={What's fair about individual fairness?},
  author={Fleisher, Will},
  booktitle={Proceedings of the 2021 AAAI/ACM Conference on AI, Ethics, and Society},
  pages={480--490},
  year={2021}
}

@article{gangbo1998optimal,
  title={Optimal maps for the multidimensional Monge-Kantorovich problem},
  author={Gangbo, Wilfrid and {S}wiech, Andrzej},
  journal={Communications on Pure and Applied Mathematics: A Journal Issued by the Courant Institute of Mathematical Sciences},
  volume={51},
  number={1},
  pages={23--45},
  year={1998},
  publisher={Wiley Online Library}
}

@article{gouic2020projection,
  title={Projection to fairness in statistical learning},
  author={Gouic, Thibaut Le and Loubes, Jean-Michel and Rigollet, Philippe},
  journal={arXiv preprint arXiv:2005.11720},
  year={2020}
}

@article{le2017existence,
  title={Existence and consistency of {W}asserstein barycenters},
  author={Le Gouic, Thibaut and Loubes, Jean-Michel},
  journal={Probability Theory and Related Fields},
  volume={168},
  pages={901--917},
  year={2017},
  publisher={Springer}
}

@article{friedler2021possibility,
  title={The (im) possibility of fairness: Different value systems require different mechanisms for fair decision making},
  author={Friedler, Sorelle A and Scheidegger, Carlos and Venkatasubramanian, Suresh},
  journal={Communications of the ACM},
  volume={64},
  number={4},
  pages={136--143},
  year={2021},
  publisher={ACM New York, NY, USA}
}

@article{hardt2016equality,
  title={Equality of opportunity in supervised learning},
  author={Hardt, Moritz and Price, Eric and Srebro, Nati},
  journal={Advances in neural information processing systems},
  volume={29},
  year={2016}
}

@inproceedings{jiang2020wasserstein,
  title={{W}asserstein fair classification},
  author={Jiang, Ray and Pacchiano, Aldo and Stepleton, Tom and Jiang, Heinrich and Chiappa, Silvia},
  booktitle={Uncertainty in artificial intelligence},
  pages={862--872},
  year={2020},
  organization={PMLR}
}

@inproceedings{silvia2020general,
  title={A general approach to fairness with optimal transport},
  author={Silvia, Chiappa and Ray, Jiang and Tom, Stepleton and Aldo, Pacchiano and Heinrich, Jiang and John, Aslanides},
  booktitle={Proceedings of the AAAI Conference on Artificial Intelligence},
  volume={34(04)},
  pages={3633--3640},
  year={2020}
}

@article{kamiran2012data,
  title={Data preprocessing techniques for classification without discrimination},
  author={Kamiran, Faisal and Calders, Toon},
  journal={Knowledge and information systems},
  volume={33},
  number={1},
  pages={1--33},
  year={2012},
  publisher={Springer}
}

@article{mccann1997convexity,
  title={A convexity principle for interacting gases},
  author={McCann, Robert J},
  journal={Advances in mathematics},
  volume={128},
  number={1},
  pages={153--179},
  year={1997},
  publisher={Elsevier}
}

@article{santambrogio2015optimal,
  title={Optimal transport for applied mathematicians},
  author={Santambrogio, Filippo},
  journal={Birk{\"a}user, NY},
  volume={55},
  number={58-63},
  pages={94},
  year={2015},
  publisher={Springer}
}

@inproceedings{zafar2017fairness,
  title={Fairness beyond disparate treatment \& disparate impact: Learning classification without disparate mistreatment},
  author={Zafar, Muhammad Bilal and Valera, Isabel and Gomez Rodriguez, Manuel and Gummadi, Krishna P},
  booktitle={Proceedings of the 26th international conference on world wide web},
  pages={1171--1180},
  year={2017}
}

@inproceedings{zemel2013learning,
  title={Learning fair representations},
  author={Zemel, Rich and Wu, Yu and Swersky, Kevin and Pitassi, Toni and Dwork, Cynthia},
  booktitle={International conference on machine learning},
  pages={325--333},
  year={2013},
  organization={PMLR}
}

@article{xu2022fair,
  author  = {Shizhou Xu and Thomas Strohmer},
  title   = {Fair Data Representation for Machine Learning at the Pareto Frontier},
  journal = {Journal of Machine Learning Research},
  year    = {2023},
  volume  = {24},
  number  = {331},
  pages   = {1--63},
  url     = {http://jmlr.org/papers/v24/22-0005.html}
}

@article{redmond2002data,
  title={A data-driven software tool for enabling cooperative information sharing among police departments},
  author={Redmond, Michael and Baveja, Alok},
  journal={European Journal of Operational Research},
  volume={141},
  number={3},
  pages={660--678},
  year={2002},
  publisher={Elsevier}
}

@article{wightman1998lsac,
  title={{LSAC} {N}ational {L}ongitudinal {B}ar {P}assage {S}tudy. {LSAC} {R}esearch {R}eport {S}eries.},
  author={Wightman, Linda F},
  year={1998},
  publisher={ERIC}
}

@phdthesis{zhou2022group,
  title={Group vs. individual algorithmic fairness},
  author={Zhou, Wanying},
  year={2022},
  school={University of Southampton}
}

@inproceedings{binns2020apparent,
  title={On the apparent conflict between individual and group fairness},
  author={Binns, Reuben},
  booktitle={Proceedings of the 2020 conference on fairness, accountability, and transparency},
  pages={514--524},
  year={2020}
}

@article{chouldechova2018frontiers,
  title={The frontiers of fairness in machine learning},
  author={Chouldechova, Alexandra and Roth, Aaron},
  journal={arXiv preprint arXiv:1810.08810},
  year={2018}
}

@article{corbett2018measure,
  title={The measure and mismeasure of fairness: A critical review of fair machine learning},
  author={Corbett-Davies, Sam and Goel, Sharad},
  journal={arXiv preprint arXiv:1808.00023},
  year={2018}
}

@article{joseph2016fairness,
  title={Fairness in learning: Classic and contextual bandits},
  author={Joseph, Matthew and Kearns, Michael and Morgenstern, Jamie H and Roth, Aaron},
  journal={Advances in neural information processing systems},
  volume={29},
  year={2016}
}

@inproceedings{lohia2019bias,
  title={Bias mitigation post-processing for individual and group fairness},
  author={Lohia, Pranay K and Ramamurthy, Karthikeyan Natesan and Bhide, Manish and Saha, Diptikalyan and Varshney, Kush R and Puri, Ruchir},
  booktitle={Icassp 2019-2019 ieee international conference on acoustics, speech and signal processing (icassp)},
  pages={2847--2851},
  year={2019},
  organization={IEEE}
}

@article{zhao2022inherent,
  title={Inherent tradeoffs in learning fair representations},
  author={Zhao, Han and Gordon, Geoffrey J},
  journal={The Journal of Machine Learning Research},
  volume={23},
  number={1},
  pages={2527--2552},
  year={2022},
  publisher={JMLRORG}
}

@inproceedings{garcia2021maxmin,
  title={Maxmin-fair ranking: individual fairness under group-fairness constraints},
  author={Garc{\'\i}a-Soriano, David and Bonchi, Francesco},
  booktitle={Proceedings of the 27th ACM SIGKDD Conference on Knowledge Discovery \& Data Mining},
  pages={436--446},
  year={2021}
}

@article{zhao2021costs,
  title={Costs and benefits of fair regression},
  author={Zhao, Han},
  journal={arXiv preprint arXiv:2106.08812},
  year={2021}
}

@inproceedings{agarwal2019fair,
  title={Fair regression: Quantitative definitions and reduction-based algorithms},
  author={Agarwal, Alekh and Dud{\'\i}k, Miroslav and Wu, Zhiwei Steven},
  booktitle={International Conference on Machine Learning},
  pages={120--129},
  year={2019},
  organization={PMLR}
}

@inproceedings{xian2023fair,
  title={Fair and optimal classification via post-processing},
  author={Xian, Ruicheng and Yin, Lang and Zhao, Han},
  booktitle={International conference on machine learning},
  pages={37977--38012},
  year={2023},
  organization={PMLR}
}
